\documentclass[11pt,reqno]{amsart}
\usepackage{bm}
\usepackage{a4wide}
\usepackage{mathtools}
\usepackage{undertilde}
\usepackage{mathabx}						
\usepackage{stmaryrd}                       
\usepackage{graphicx}
\usepackage{marginnote}
\usepackage{tikz}

\usepackage{subcaption}

\newtheorem{thm}{Theorem}

\newtheorem{lem}{Lemma}

\newtheorem{pb}{Problem}

\newtheorem{defi}{Definition}

\newcommand\Q{\mathbb{Q}}
\newcommand\R{\mathbb{R}}
\newcommand\Z{\mathbb{Z}}
\newcommand\N{\mathbb{N}}
\newcommand\Sph{\mathbb{S}}
\newcommand\T{\mathbb{T}}

\begin{document}

\large

\title[Danzer et Dense Forests]{Around the Danzer Problem and the Construction of Dense Forests.}
\author{Faustin Adiceam}
\address{F.A. : Department of Mathematics, The University of Manchester, UK} \email{ faustin.adiceam@manchester.ac.uk}
\thanks{FA's research was supported by EPSRC Grant EP/T021225/1. He would like to thank Timothy Gowers, Yaar Solomon, Ioannis Tsokanos and Barak Weiss for detailed comments on prelimary drafts of this survey.}

\begin{abstract}
A 1965 problem due to Danzer asks whether  there exists a set with finite density in Euclidean space intersecting any convex body of volume one. A suitable weakening of the volume constraint leads one to the (much more recent) problem of constructing \emph{dense forests}. These are discrete point sets becoming uniformly close to long enough line segments. 

Progress towards these problems have so far involved a wide range of ideas surrounding areas as varied as combinatorial and computation geometry,  convex geometry, Diophantine approximation, discrepancy theory, the theory of dynamical systems, the theory of exponential sums, Fourier analysis, homogeneous dynamics, the mathematical  theory of quasicrystals and probability theory.

The goal of this paper is to survey the known results related to the Danzer Problem and to the construction of dense forests, to generalise some of them and to state a number of open problems to make further progress  towards this  longstanding question.

\end{abstract}

\maketitle


\begin{center}
\emph{A Philippe Gavila.}
\end{center}

\setcounter{tocdepth}{2}
\tableofcontents

\section{The Danzer Problem}

\subsection{Statement}
 
Let $n\ge 2$ be an integer standing for a dimension. Throughout, $B_2\left(\bm{x}, T\right)$ denotes the Euclidean ball centered at $\bm{x}\in\R^n$ with radius $T\ge 0$. The \emph{density} of a subset $\mathfrak{D}\subset\R^n$ is defined as the quantity 
\begin{equation*}
d_n\left(\mathfrak{D}\right)\;=\; \limsup_{T\rightarrow\infty} \frac{\#\left(\mathfrak{D}\cap B_2\left(\bm{0}, T\right)\right)}{T^n}\cdotp
\end{equation*}
 
The following open problem, due to Ludwig Danzer (1927--2011), and initially concerned with the planar case $n=2$ only, seems to have been posed first on the occasion  of a colloquium on the topic of convexity held at the university of Copenhagen in 1965 --- see~\cite{Da65}.

\begin{pb}[Danzer, 1965]\label{danzer}
Does there exists a set with finite density intersecting any convex body of volume 1 in dimension $n\ge 2$?
\end{pb}

The property of intersecting each convex body with volume 1 will from now on be referred to as the \emph{Danzer property}. Clearly, if $\mathfrak{D}$ satisfies this property, then the dilated set $c\mathfrak{D}=\left\{c\cdot\bm{d}\; :\; \bm{d}\in\mathfrak{D}\right\}$, where $c>0$, intersects any covex body with volume $c^n$. No particular importance must therefore be attached to the requirement that the volumes considered in Problem~\ref{danzer}  must be unity~: they can take any fixed, strictly positive value. This observation shows that the Danzer Problem can be restated as follows, as is done in~\cite[p.148]{CFG} for instance~: \emph{given a subset of $\R^n$ with finite density, does there  exist convex sets of arbitrarily large volumes not intersecting it?} 

It is not clear whether the Danzer problem can be answered in the affirmative. In~\cite{branetal} for instance, it is conjectured that there exists no set satisfying the assumptions of Problem~\ref{danzer}.\\

Problem 9 in~\cite[p.285]{BC}  is concerned with a question dual to Danzer's~:

\begin{pb}[The Dual Danzer Problem]\label{dualdanzer}
Let $M\ge 2$ be an integer and let $\mathfrak{D}\subset\R^n$. Assume that  any convex body with unit volume contains at most $M$ points of $\mathfrak{D}$. Does this imply that $\mathfrak{D}$ has zero density?
\end{pb}

Most of this survey will be concerned with Danzer's Problem. Its dual version, about which very little seems to be known, will be mentioned only briefly.

\subsection{Some  other Statements Equivalent to Danzer's Problem}\label{equistatesec}

As implicitly observed by Bambah and Woods~\cite[\S 3]{BW} and then explicitly stated by Solomon and Weiss~\cite[Proposition 5.4]{SW}, there is no loss of generality in restricting the statement of Problem~\ref{danzer} to the set of \emph{boxes} with unit volume (a box is a parallelepiped with orthogonal adjacent faces). 

This indeed follows from a theorem due to F.~John asserting that any convex set $\mathcal{K}\subset\R^n$ contains a unique ellipsoid $\mathcal{E}$ of maximal volume, which furthermore satisfies the property that $\mathcal{K}\subset \sqrt{n}\cdot\mathcal{E}$ --- see~\cite[pp.13--16]{ball}. Consequently, there exists a constant $\alpha_n>0$ such that any convex body $\mathcal{K}\subset\R^n$ contains a box $\mathcal{R}_1$ and is contained in a box $\mathcal{R}_2$ with the property that $\textrm{vol}(\mathcal{R}_2)/\textrm{vol}(\mathcal{R}_1)\le \alpha_n$ --- see  Figure~\ref{figjohnthm}. If one assumes that $\mathcal{K}$ has unit volume, this implies that the box $\mathcal{R}_1$ has volume at least $s=\alpha_n^{-1}$. Danzer's Problem can therefore be reduced to proving the existence of a point in the smaller set $\mathcal{R}_1$. Without loss of generality in view of the remark on the volume constraint made after the statement of Problem~\ref{danzer}, the box $\mathcal{R}_1$ can furthermore be assumed to have unit volume, which justifies the above claim.

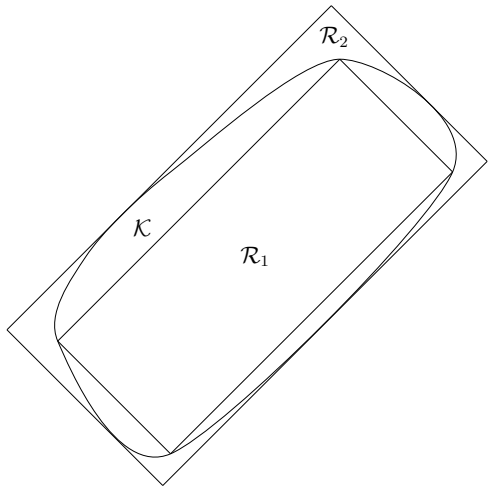
\begin{figure}[h!]
\begin{center}
\scalebox{.75}{
  \begin{tikzpicture}    
    \draw (0,0)--(5,5);
    \draw (-2,2)--(3,7);
    \draw (0,0)--(-2,2);
    \draw (5,5)--(3,7);
    
    \node (p1) at (0, 0) {};
    \node (p2) at (5, 5) {};
    \node (p3) at (3, 7) {};
    \node (p4) at (0,5) {};
    \node (p5) at (-1.5,3.5) {};
    \node (p6) at (-2, 2) {};
    \draw plot [smooth cycle,tension=0.5, ultra thick] coordinates {(p1) (p2) (p3) (p4)(p5)(p6)};
        
    \draw (-2.9,2.2)--(2.85,7.95) ;
    \draw (2.85,7.95)-- (5.6125,5.1875);
    \draw (5.6125,5.1875)--(-0.1375,-0.5625);
    \draw (-0.1375,-0.5625)--(-2.9,2.2);
    
    \node  at (1.5, 3.5) {$\mathcal{R}_1$};
    \node  at (-0.5, 4) {$\mathcal{K}$};
    \node at (2.9, 7.4) {$\mathcal{R}_2$};        
  \end{tikzpicture}
  }
  \end{center}
  \caption{Illustration of John's Theorem in the plane~: the convex body $\mathcal{K}$ contains an inner rectangle $\mathcal{R}_1$ and is contained in an outer rectangle $\mathcal{R}_2$ such that the ratio $\textrm{area}(\mathcal{R}_2)/\textrm{area}(\mathcal{R}_1)$ is bounded above by an absolute constant.} 
  \label{figjohnthm}
\end{figure}

Any unit volume box can be obtained as the image of the standard box $\left[0,1\right]^n$ by an element of  the orientation and volume preserving affine group 
\begin{equation*}\label{defgn}
G_n\; =\; SL_n\left(\R\right)\ltimes \R^n.
\end{equation*} 
As noticed in~\cite[Proposition 3.1]{SW}, this observation provides a natural reformulation of Problem~\ref{danzer} in terms of a property satisfied by a group action~:

\begin{pb}[Statement Equivalent to the Danzer Problem]\label{equivdanzer}
Does there exist a set $\mathfrak{D}\subset\R^n$ with finite density such that for any $g\in G_n$, the condition $$g\left([0,1]^n\right)\cap \mathfrak{D}\neq \emptyset$$ holds?
\end{pb}

A topological reformulation of Danzer's Problem also appears in~\cite[Lemma 2.3]{SSW}. It is based on the Chabauty--Fell topology on the space $\mathcal{X}_n$ of all closed subsets in $\R^n$. Recall that this topology is induced by the metric $d$ defined as follows~: given $F_1, F_2\in \mathcal{X}_n$, set 
\begin{equation}\label{fell}
d\left(F_1, F_2\right)\;=\; \min\left\{1,  \Delta\left(F_1, F_2\right)  \right\},
\end{equation} 
where $\Delta\left(F_1, F_2\right)$ is the infimum over all $\epsilon>0$ such that $$F_1\cap B_2\left(\bm{0}, \epsilon^{-1}\right)\subset \bigcup_{\bm{x}\in F_2}B_2\left(\bm{x}, \epsilon \right) \qquad \textrm{and} \qquad F_2\cap B_2\left(\bm{0}, \epsilon^{-1}\right)\subset \bigcup_{\bm{x}\in F_1}B_2\left(\bm{x}, \epsilon \right).$$

It is known that the space $\left(\mathcal{X}_n, d\right)$ is a complete metric space. Furthermore, the affine action of $G_n$ on $\R^n$ induces a natural $G_n$--action on $\mathcal{X}_n$. A short verification based on the reduction made before the statement of Problem~\ref{equivdanzer}  then yields to the following (see~\cite[p.6588]{SSW} for details)~:

\begin{pb}[Statement Equivalent to the Danzer Problem]\label{equivdanzerbis}
Does there exist a set $\mathfrak{D}\subset\R^n$ with finite density such that the empty set does not belong to the closure, in the Chabauty--Fell topology, of the set $G_n\cdot\mathfrak{D}$ (i.e.~such that $\emptyset\not\in\overline{G_n\cdot\mathfrak{D}}$)?
\end{pb}

Problem~\ref{equivdanzerbis} thus provides a natural link between the Danzer Problem  and the pro\-per\-ties of the dynamical system $\left(\mathcal{X}_n, G_n\right)$.

\subsection{The State of Art}\label{statartdanzer} Attempts to tackle Danzer's question  have  focused in three different directions~: firstly, by relaxing the density constraint. The known results in this direction are reviewed below. Secondly, by relaxing in a relevant way the volume constraint. This approach is very recent and leads one to the topic of dense forests which \S~\ref{secforests} is concerned with. Thirdly, and this is the subject of \S\ref{seccandidate} below, by showing that classes of discrete points sets enjoying some additional structure do not satisfy the Danzer property despite having a finite density. \\

When relaxing the density constraint, Bambah and Woods in their seminal 1971 paper~\cite{BW} showed that a union of grids (that is, a union of translated lattices) can be used to define constructively a  set $\mathfrak{D}\subset\R^n$ satsifying the Danzer property whose density fails to be finite up to a power of a logarithm. More precisely, they obtain the density bund $\#\left(\mathfrak{D}\cap B_2\left(\bm{0}, T\right)\right) = O\left(T^n(\log T)^{n-1}\right)$ for $n\ge 2$. 

The work by Bambah and Woods has long remained the only significative contribution to Danzer's Problem until the 2016  paper by Solomon and Weiss~\cite{SW}. It is proved therein that one can find a set satisfying the required intersection property with a better growth rate for its density~: 

\begin{thm}[Solomon \& Weiss, 2016]\label{solweibest} There exists a set $\mathfrak{D}\subset\R^n$ satisfying the Danzer property such that $\#\left(\mathfrak{D}\cap B_2\left(\bm{0}, T\right)\right) = O\left(T^n \log T\right)$.
\end{thm}

This stands as the best known result when relaxing the density constraint in Danzer's Problem. The construction leading to the above statement  is of a probabilistic and combinatorial nature and is related to the so--called Danzer--Rogers Problem presented in \S~\ref{danzerrogers}. The following curious consequence is derived from this combinatorial approach (see~\cite[Corollary 1.5]{SW} for details)~:

\begin{thm}[Solomon \& Weiss, 2016]
If a Danzer set exists in $\R^n$, then there exists a Danzer set contained in $\Q^n$.
\end{thm} 

As for positive results when keeping both the density and the volume constraints, one most notable contribution is presented in~\cite{SS} and relies on the equivalent formulation of the Danzer Problem stated in Problem~\ref{equivdanzer}~:

\begin{thm}[Simmons \& Solomon, 2016]\label{simsol}
There exists a set $\mathfrak{D}_n\subset\R^n$ with finite density intersecting any aligned box with volume 1.
\end{thm}

Here, a box is said to be \emph{aligned} if it is of the form $\prod_{i=1}^n\left[a_i, b_i\right]$; that is, if it is a box with edges parallel to the coordinate axes. The construction provided by Simmons and Solomon is simple enough to be described here and relies on ideas well--known in discrepancy and lattice theories. When $n=2$, it is defined from a variant of the low discrepancy Van der Corput sequence; namely, 
\begin{equation}\label{defD2}
\mathfrak{D}_2\;=\; \left\{\left(\pm\sum_{n\in\Z}a_n2^n,\, \pm\sum_{n\in\Z}a_n2^{-n}\right)\in\R^2\; :\; \left(a_n\right)_{n\in\Z} \in \left\{0,1\right\}^{\Z}_{fin}\right\},
\end{equation}
where $\left\{0,1\right\}^{\Z}_{fin}$ stands for the set of sequences indexed by $\Z$, taking values in $\{0,1\}$, and with finitely many nonzero terms. When $n\ge 3$, a different approach leads the authors to define $\mathfrak{D}_n$ as any \emph{admissible lattice}, constructing along the way explicit examples of such lattices from an embedding in $\R^n$ of the ring of integers of a totally real number field of degree $n$. Here, a lattice $\Lambda\subset\R^n$ is said to be admissible if $$\inf_{A\in \Delta_n\left(\R\right)}\;\inf_{\bm{v}\in\Lambda\backslash\left\{\bm{0}\right\}}\;\left\|A\bm{v}\right\|_2\;>\; 0, $$ where $\Delta_n\left(\R\right)$ denotes the subgroup of $SL_n\left(\R\right)$ made of those diagonal matrices with positive entries.

As noted by Simmons and Solomon in their paper, these sets $\mathfrak{D}_n$ are nevertheless not Danzer sets. Moreover, it follows from a result of Bambah and Woods~\cite[Theorem 1]{BW} that when $n\ge 3$, any  finite union of affine images of the set $\mathfrak{D}_n$ (which is thus a finite union of grids) cannot be a Danzer set either (see~\S \ref{secgrid} for further details). A similar conclusion should be expected in the case $n=2$ but is not known~:

\begin{pb}
Can a finite union of affine images of the set $\mathfrak{D}_2$ defined in~\eqref{defD2} be a Danzer set in $\R^2$?
\end{pb}
\subsection{To go further}\label{togosec} Simmons and Solomon's Theorem~\ref{simsol} can be seen as a particular case of a more general problem which we now describe.

A unit volume box in $\R^n$ can be obtained as the image of the standard box $[0,1]^n$ by a  transformation $g'\in G_n$  that can be  decomposed uniquely as the product
\begin{equation}\label{decomp}
g'=tka.
\end{equation} 
Here, $t$ stands for a translation, $k$ for an element of the special orthognal group and $a$ for a linear mapping defined by a diagonal matrix with positive entries and determinant one (this is an immediate consequence of the Iwasawa decomposition in $SL_n(\R)$)
Denote by $G'_n$ the subset of the orientation and volume preserving affine group $G_n$ containing all those elements $g'$ that can be decomposed  as~\eqref{decomp}.

\begin{pb}\label{pbmeas}
Assume that $G'_n$ is equipped with a measure $\nu_n$, and let $\eta\ge 0$. Does there exist a set $\mathfrak{D}(\eta)\subset\R^n$ with finite density such that the measure of the set of unit volume boxes \emph{not} intersecting $\mathfrak{D}(\eta)$  is at most $\eta$? More formally, does the inequality $$\nu_n\left(\left\{g'\in G'_n\; :\; g'\left([0,1]^n\right)\cap\mathfrak{D}(\eta)=\emptyset\right\}\right)\;\le\; \eta$$ hold?
\end{pb}

Of course, Problem~\ref{pbmeas} only becomes relevant when defining a specific measure $\nu_n$. In this respect, a natural choice would be to choose $\nu_n$ as the product of the Haar measures on each of the components $t$, $k$ and $a$. It should then be noted that \emph{the case $\eta=0$ is nothing but a reformulation of the Danzer Problem}~: indeed, this  follows from the observation that a convex body not intersecting a finite density (hence discrete) set will also necessarily miss a neighbourhood of this set. 

This observation can also be used to recast Simmons and Solomon's result in the framework of Problem~\ref{pbmeas}. To see this, take $\nu_n$ to be the product of the Haar measures on the $t$ and $a$ components in the decomposition~\eqref{decomp}, and of the Dirac mass at the identity on the $k$ component. Theorem~\ref{simsol} is then equivalent to a positive answer to Problem~\ref{pbmeas} in the case $\eta=0$.

\section{Some Problems (closely) related to Danzer's} 

\subsection{The Danzer--Rogers Problem}\label{danzerrogers}

The construction of a set satisfying the Danzer property with the best known density growth in Theorem~\ref{solweibest} is achieved thanks to an equivalence between Danzer's Problem and one of a more combinatorial nature~: the Danzer--Rogers Problem. To describe this approach, we first introduce some definitions. 

A \emph{Range Space} is a pair $\left(N, \mathcal{R} \right)$, where $N$ is a set whose elements are referred to as \emph{points}, and  $\mathcal{R}$ is a subset of the power set $\mathcal{P}(N)$ whose elements are referred to as \emph{ranges}. Given a probability measure $\mu$  on $N$ and given $\epsilon>0$, a subset $N_{\epsilon}\subset N$ is said to be an \emph{$\epsilon$--net} if $N_{\epsilon}$ intersects non trivially any range $S$ as soon as $\mu\left(S\cap N\right)\ge \epsilon$. Note that when $N$ is finite with $n$ elements and when $\mu$ is chosen as the counting measure, the latter condition can be reformulated as $\#\left(S\cap N\right)\ge \epsilon n$.

We now state \cite[Theorem 1.4]{SW}~: 

\begin{thm}[Solomon \& Weiss, 2016] \label{equivdanzerog}
Let $g~:\R_+\rightarrow \R_+$ be such that $x\mapsto g(x)/x^n$ is non--decreasing. Assume that there exists $C>0$ such that  $g(2x)\le C \cdot g(x)$ for all $x\in\R_+$. In the above notation, set $N=[0,1]^n$,  equip it with the Lebesgue measure and choose $\mathcal{R}$ as the set of all boxes contained in $[0,1]^n$. 

Then, \sloppy there exists a subset $\mathfrak{D}\subset\R^n$ enjoying the Danzer property with density growing like $\#\left(\mathfrak{D}\cap B_2\left(\bm{0}, T\right)\right) = O\left(g(T)\right)$ if, and only if, for any $\epsilon>0$, there exists an $\epsilon$--net $N_{\epsilon}$ in the Range Space $\left(N, \mathcal{R}\right)$ which has a finite cardinality growing like $\# N_{\epsilon}=O\left(g\left(\epsilon^{-1/n}\right)\right)$.
\end{thm}

The case $g(x)=x^n$  provides an equivalence between the Danzer Problem and the following open problem in combinatorial and computational geometry which is conjectured not to hold true in~\cite[\S 6]{RT}~:

\begin{pb}[The Danzer--Rogers Problem --- Statement Equivalent to the Danzer Problem]\label{danzrogpb}
Consider the Range Space $\left(N, \mathcal{R}\right)$, where $N=[0,1]^n$ is equipped with the Lebesgue measure, and where $\mathcal{R}$ is the set of boxes contained in $[0,1]^n$. Given $\epsilon>0$, does there exist an $\epsilon$--net $N_{\epsilon}$ such that $\# N_{\epsilon}=O\left(\epsilon^{-1}\right)$?
\end{pb}

To be precise, the classical statement of the Danzer--Rogers problem (see~\cite{BradC, PT} and the references therein) is concerned with the case $n=2$ and the case where the range space is the set of all convex bodies. However, here again, there is no loss of generality in restricting oneself to the set of boxes.  It is furthermore easy to see that the lower bound $\# N_{\epsilon}\ge \epsilon^{-1}-1$ always holds. Indeed, given $k$ points in $[0,1]^n$, consider the hyperplanes containing these points and all parallel to a given $(n-1)$--dimensional face of the hypercube. These hyperplanes then partition $[0,1]^n$ into $k+1$ boxes, at least one of which has area larger than $1/(k+1)$, whence the claim. \\

Theorem~\ref{solweibest} by Solomon and Weiss is obtained by combining Theorem~\ref{equivdanzerog} with the construction of an $\epsilon$--net with the best known growth for its cardinality. This construction is due to Haussler and Welzl~\cite{HW}, is of a purely probabilistic nature, and provides the bound $\# N_{\epsilon}=O\left(\epsilon^{-1}\log\left(\epsilon^{-1}\right)\right)$. The following question is of interest in order to obtain a deterministic construction~:

\begin{pb}
Under the assumptions of the Danzer--Rogers Problem, does there exist a \emph{deterministic} $\epsilon$--net $N_{\epsilon}$ such that $\# N_{\epsilon}=O\left(\epsilon^{-1}\log\left(\epsilon^{-1}\right)\right)$?
\end{pb}

The difficulty of the Danzer--Rogers Problem can be seen through the following counterexamples when slightly modifying its assumptions. In~\cite{PT}, Pach and Tardos show that the answer to Problem~\ref{danzrogpb} is negative in the planar case $n=2$ when the range space is extended to the class of so--called quasi--boxes, which are essentially obtained by ``bending'' one of the axes of a box by a bounded angle (see~\cite[p.6]{PT} for  a precise de\-fi\-ni\-tion). Also, Alon notes in~\cite[p.6]{Alon} that there are probability distributions on $[0,1]^2$, other than the Lebesgue measure, for which the answer to Problem~\ref{danzrogpb} is negative.

The Danzer--Rogers problem nevertheless admits a positive solution when the range space $\mathcal{R}$ is the set of all aligned boxes in $[0,1]^n$~: this follows upon combining the proof of Theorem~\ref{equivdanzerog} and Theorem~\ref{simsol}.

\subsection{The Boshernitzan--Conway Problem} Since it is not clear whether the Danzer Problem can be answered in the affirmative, studying a stronger version of the problem 
provides a way to test the limits of the assumptions under which it could be true. Because there is already no loss of generality in restricting the set of covex bodies to the set of boxes, one can look for a condition stronger than that of finite density. In this respect, recall the following hierarchy between discrete sets : $$\textrm{uniformly discrete} \Rightarrow \textrm{finite upper uniform density} \Rightarrow  \textrm{finite density} \Rightarrow \textrm{discrete}.$$
Here, a set $\mathfrak{D}\subset\R^n$ is said to be \emph{uniformly discrete} if the distance between any two distinct points in $\mathfrak{D}$ is uniformly bounded below by a strictly positive constant; that is, if  $$\inf_{\underset{\bm{x}\neq\bm{y}}{\bm{x}, \bm{y}\in\mathfrak{D}}}\; \left\|\bm{x}-\bm{y}\right\|_2\;>\; 0. $$
It is said to have a finite \emph{upper uniform density} if 
\begin{equation}\label{defupperfinitedensity}
\limsup_{T\rightarrow\infty}\; \sup_{\bm{x}\in\R^n}\; \frac{\#\left(\mathfrak{D}\cap B_2\left(\bm{x}, T\right)\right)}{T^n}\;<\; \infty.
\end{equation}

The following problem, with assumptions stronger than those of Problem~\ref{danzer}, is independently due to Boshernitzan and Conway. Neither of them seemed to have been aware of the question posed by Danzer. For reasons narrated in~\cite{biocon}, Conway coined it the ``Dead Fly Problem''. 

\begin{pb}[The Boshernitzan--Conway Problem]\label{boshcon}
Does there exist a \emph{uniformly discrete} set intersecting any convex body with volume 1 in $\R^n$?
\end{pb}

This problem is open, although the relaxation of the volume constraint through the concept of dense forests provides sets very close to achieving the desired conclusion (in a suitable sense). This will be detailed in \S~\ref{stateartdenseforest}.

\subsection{The Gowers Problem} Gowers suggested in~\cite{Gowers}  a version of Danzer's Problem~\ref{danzer} stronger than the orginal one which, in substance, is obtained by  combining some of the assumptions of its dual version\footnote{In private communication, Timothy Gowers nevertheless indicated to the author that he was not aware of the statement of the Dual Danzer Problem when posing Problem~\ref{gowerpb}.} (Problem~\ref{dualdanzer}). This was recently solved in~\cite{SSW}.

\begin{pb}[Gowers, 2000 --- solved]\label{gowerpb}
Does there exists a set $\mathfrak{D}\subset\R^n$  and an integer $m\ge 1$ such that for any convex body $K\subset\R^n$ with unit volume, it holds that $1\le \#\left(K\cap\mathfrak{D}\right)\le m$?
\end{pb}

\begin{thm}[Solan, Solomon \& Weiss, 2017]\label{mainssw}
Gowers' problem admits a negative answer~: if $\mathfrak{D}\subset\R^n$ satisfies the Danzer property, then for any $\epsilon>0$ and any $k\ge 1$, there exists an ellipsoid $E_k$ with volume less than $\epsilon$ such that $\#\left(E_k\cap\mathfrak{D}\right)\ge k$.
\end{thm}

Two proofs are provided in~\cite{SSW} for this statement~: one which is constructive and based on an induction on $k\ge 1$, and another one of a dynamical nature. 

The dynamical proof relies on a reformulation of the Gowers property in terms of the dynamical system $\left(\mathcal{X}_n, G_n\right)$ introduced in \S\ref{equistatesec}. More precisely, it is shown in~\cite[Lemma 2.3]{SSW} that a set $\mathcal{G}\subset\R^n$ is a solution to the Gowers Problem if, and only if,  two conditions are simultaneously satisfied~: $\emptyset\not\in \overline{G_n\cdot\mathcal{G}}$ and $\R^n\not\in \overline{G_n\cdot\mathcal{G}}$ (the closures are here again taken with respect to the Chabauty--Fell topology). This should be compared with the statement of  Problem~\ref{equivdanzer}. The authors then show that at least one of these relations cannot hold.

As for the constructive proof, it provides a quantitative bound for the diameter of the ellipsoid $E_k$ satisfying the conclusion of Theorem~\ref{mainssw}. This is of interest in view of the following reformulation of Gowers' Problem in terms of $\epsilon$--nets.

\begin{pb}[Statement Equivalent to Gowers' Problem --- solved]\label{enetgowers}
Consider the Range Space $\left(N, \mathcal{R}\right)$, where $N=[0,1]^n$ is equipped with the Lebesgue measure, and where $\mathcal{R}$ is the set of boxes contained in $[0,1]^n$. Let $m\ge 1$. Given $\epsilon>0$, does there exist an $\epsilon$--net $N_{\epsilon}$ such that for any range $S$, it holds that  $1\le \# \left(N_{\epsilon}\cap S\right)\le m$?
\end{pb}

The constructive proof of Theorem~\ref{mainssw} is adapted in~\cite[Theorem 1.4]{SSW} to provide a quantitative refutation to Problem~\ref{enetgowers}~:

\begin{thm}[Solan, Solomon \& Weiss, 2017]\label{enetgowerssol}
Problem~\ref{enetgowers} admits a  negative answer~: for every $\epsilon>0$, \sloppy there exists a range $S$ of volume at least $\epsilon$ such that 
\begin{equation}\label{lowerbound}
\# \left(N_{\epsilon}\cap S\right) = \Omega\left(\log\log\left(\epsilon^{-1}\right)\right).
\end{equation}
\end{thm}

Here, the notation in~\eqref{lowerbound} means that  the lower limit, as $\epsilon\rightarrow 0^+$, of the ratio between the functions on the left--hand side and the right--hand side is strictly positive.

To be precise, Problem~\ref{enetgowers} and Theorem~\ref{enetgowerssol} are stated in~\cite{SSW} in the case that the range space is the set of convex bodies in $[0,1]^n$. In view of John's Theorem (see \S\ref{equistatesec}), there is, here again, no loss of generality in restricting the consideration to the smaller set of boxes. 
Two natural problems emerge from Theorem~\ref{enetgowerssol}. The first one is essentially a particular case of~\cite[Question 5.6]{SSW} and second one is~\cite[Question 5.7]{SSW}.

\begin{pb}
What are the range spaces $\left(N, \mathcal{R}\right)$, where $N=[0,1]^n$ is equipped with the Lebesgue measure, for which Problem~\ref{enetgowers} admits a positive solution for some $m\ge 1$?
\end{pb}

A partial answer to this question is provided in~\cite[Proposition 5.5]{SSW}~:  Problem~\ref{enetgowers} admits a positive solution for some $m\ge 1$ when $\mathcal{R}$ is the set of aligned boxes. This is a direct consequence of Theorem~\ref{simsol}.

\begin{pb}[Nati Linial]
Can the lower bound in~\eqref{lowerbound} be improved by replacing it with a function $f~: \R_+\rightarrow\R_+$ tending to infinity as $\epsilon\rightarrow 0^+$ faster than $\Omega\left(\log\log\left(\epsilon^{-1}\right)\right)$? For instance, can one take $f\left(\epsilon\right)=\Omega\left(\log\left(\epsilon^{-1}\right)\right)$?
\end{pb}

\section{Dense Forests and their Relations with the Danzer Problem}\label{secforests}

\subsection{Visibility in a Dense Forest} Bishop's investigations  in a problem of rectifiability of curves led him to the following problem stated in~\cite[\S 2.3]{Bishop}~: how must the trees be positioned in a forest if they all have the same radius, say $\epsilon>0$, and if their centers are never closer than one unit, to fullfil this property~:  no matter where an observer stands and what direction he looks in, he will never be able to see further than some finite distance $\mathcal{V}$? What is the smallest value of the visibility $\mathcal{V}$ that can be achieved this way in terms of $\epsilon$?

This is an instance of a visibility problem, the most well--known of which is probably  P\'olya's orchard problem posed in~\cite{Polya})~: how thick  must be the trunks of the trees in a regularly spaced circular orchard grow  if an observer should not be able to see the horizon when standing in the center? The literature on visibility problems is abundant, and the reader is referred to the papers~\cite{Adiceam, Adiceambis, BS} and to the references within for further details and accounts on the latest developments in this area.

The fundamental difference between P\'olya's and Bishop's problems is that in the former, the observer is standing at the origin whereas in the latter, he is allowed to stand anywhere in the space. This makes Bishop's question a lot harder, and this is also what relates it to Danzer's Problem. \\

To make this relation explicit, we state the following definition, which can be seen as a formalisation of Bishop's question motivated by the statement of Problem~\ref{danzer}~:
\begin{defi}
A \emph{dense forest} is a subset $\mathfrak{F}\subset\R^n$ which admits a function $\mathcal{V}~: \epsilon>0 \mapsto \mathcal{V} (\epsilon)\ge 0$ defined in a neighbourhood of the origin such that for any $\epsilon>0$ small enough, any line segment with length $\mathcal{V}(\epsilon)$ comes $\epsilon$--close to a point in $\mathfrak{F}$. In other words, the following holds for all $\epsilon >0$ small enough~: 
\begin{equation}\label{defdenseforest}
\forall  \bm{\alpha}\in\Sph^{n-1}, \;\;\; \forall  \bm{\beta}\in\R^{n},  \;\;\; \exists t\in [0, \mathcal{V}(\epsilon)], \;\;\; \exists \bm{f}\in\mathfrak{F}, \;\;\; \left\|\bm{\beta}+t\bm{\alpha}-\bm{f}\right\|_{\infty} < \epsilon.
\end{equation}
The function $\mathcal{V}$ is referred to as a {\em visibility function} for $\mathfrak{F}$.
\end{defi}

In this definition, $\left\|\; .\; \right\|_{\infty}$ stands for the sup norm and $\Sph^{n-1}$ denotes the Euclidean sphere in dimension $n$. By considering a packing of the hypercube with sidelength $\epsilon^{-1}$ centered at the origin by translates of the box $ \left[0, \epsilon\right]^{n-1}\times\left[0, \mathcal{V}\left(\epsilon\right)\right]$, the finite density condition implies that a visibility function in an $n$--dimensional dense forest admits a lower bound  of the form $\mathcal{V}\left(\epsilon\right)\ge c\cdot\epsilon^{-(n-1)}$ for some $c>0$ and all $\epsilon>0$. This observation prompts a natural question~:

\begin{pb}[Main Problem related to the Concept of Dense Forests]\label{mainpbdenseforest}
Does there exist a dense forest with visibility $\mathcal{V}\left(\epsilon\right)=O\left(\epsilon^{-(n-1)}\right)$ in dimension $n\ge 2$?
\end{pb}

Here and throughout, given two functions of a single variable $f$ and $g$, the notation $f=O\left(g\right)$ means that there exists $C>0$ such that $f(x)\le C\cdot g(x)$ for all values of the variable $x$.\\

The construction of dense forests can be seen as a suitable relaxation of the volume constraint in Danzer's Problem. Indeed, asking that a line segment should get $\epsilon$--close to a point is the same as asking that the $\epsilon$--neighbourhood of this line segment  should contain this point (this neighbourhood is measured in terms of the sup norm according to the above definition). Such a neighbourhood contains an inner box which has a long--edge with length $\mathcal{V}\left(\epsilon\right)$ and $n-1$ short edges, all with lengths at least $\epsilon$ (in other words, the inner box can be taken as  the image, under an orthognal transformation and a translation, of the box $\left[0, \mathcal{V}\left(\epsilon\right)\right]\times \left[0, \epsilon\right]^{n-1}$). The neighbourhood is furthermore contained in an outer box obtained by dilating the inner one by a factor depending only on the dimension. See Figure~\ref{neigfres} for an illustration in the planar case.

\begin{figure}[h!]
\begin{center}
\scalebox{1}{

\begin{subfigure}{.45\textwidth}
  \centering

  \begin{tikzpicture}    
  
  	\draw[ultra thick] (2-0.5, 2.25-0.5)-- (2+0.5, 2.25-0.5)-- (2+0.5, 2.25+0.5) -- (2-0.5, 2.25+0.5)--cycle;
  	\draw[ultra thick] (0,0)--(5,5);
  	\draw plot[mark=x, mark options={color=black, scale=2,  thick}] coordinates {(2, 2.25)};
  	\draw[<->] (2,2.25)--(1.5, 2.25);
	\node [above] at (1.75,2.2) {$\epsilon$};    	
  	\draw[<->] (0+0.3536,0-0.3536)--(5+0.3536,5-0.3536);
  	\node [below right] at (2.5+0.3536,2.5-0.3536) {$\mathcal{V}\!\left(\epsilon\right)$};   
  		
  \end{tikzpicture}

  \label{ptcloseseg}
\end{subfigure}
\hfill
\begin{subfigure}{.45\textwidth}
  \centering

  \begin{tikzpicture}    
  
 	\fill[gray!20]  (-0.3536, -0.3536)--(0-0.5,0+0.5)  -- (0-0.5,0+0.5)--(5-0.5,5+0.5) -- (5-0.5,5+0.5)--(5+0.3536, 5+0.3536)--(5+0.5,5-0.5)--(0+0.5,0-0.5)--cycle; 
  
  	\draw[ultra thin] (2-0.5, 2.25-0.5)-- (2+0.5, 2.25-0.5)-- (2+0.5, 2.25+0.5) -- (2-0.5, 2.25+0.5)--cycle;
  	\draw[ultra thin] (0,0)--(5,5);
  	\draw plot[mark=x, mark options={color=black, scale=2,  thick}] coordinates {(2, 2.25)};    	
  	\draw[ultra thick] (0+0.5,0-0.5)--(5+0.5,5-0.5);
  	\draw[ultra thick] (0-0.5,0+0.5)--(5-0.5,5+0.5);
  	\draw[ultra thick] (5-0.5,5+0.5)--(5+0.3536, 5+0.3536);  	  		
  	\draw[ultra thick] (5+0.5,5-0.5)--(5+0.3536, 5+0.3536);   	  		
  	\draw[ultra thick] (-0.3536, -0.3536)--(0+0.5,0-0.5);  
  	\draw[ultra thick] (-0.3536, -0.3536)--(0-0.5,0+0.5);
  	\draw[<->] (-0.3536+0.7,-0.3536-0.7)--(5+0.3536+0.7,5+0.3536-0.7);  	
  	\node [below right] at (2.5+0.7,2.5-0.7) {$\mathcal{V}\!\left(\epsilon\right)+\sqrt{2}\epsilon$};   
  	\draw[<->] (5+0.5+0.5,5+0.5-0.5)--(5+0.5-0.5,5+0.5+0.5);  		
  	\node [above right] at (5+0.4,5+0.4) {$\sqrt{2}\epsilon$};     	  	
  \end{tikzpicture}

  \label{neigseg}
\end{subfigure}

}
\end{center}
  \caption{On the left, a point in $\R^2$ lying $\epsilon$--close (with respect to the sup norm) to a line segment with length $\mathcal{V}\!\left(\epsilon\right)\ge 1$. On the right, the $\epsilon$--neighbourhood of the line segment is shaded in gray. The point under consideration lies in it. This neighbourhood contains a rectangle with side--lengths $\sqrt{2}\epsilon$ and $\mathcal{V}\!\left(\epsilon\right)$ and is contained in a rectangle with side--lengths $\sqrt{2}\epsilon$  and $\mathcal{V}\!\left(\epsilon\right)+\sqrt{2}\epsilon\le 2\mathcal{V}\!\left(\epsilon\right)$.}
  \label{neigfres}
\end{figure}
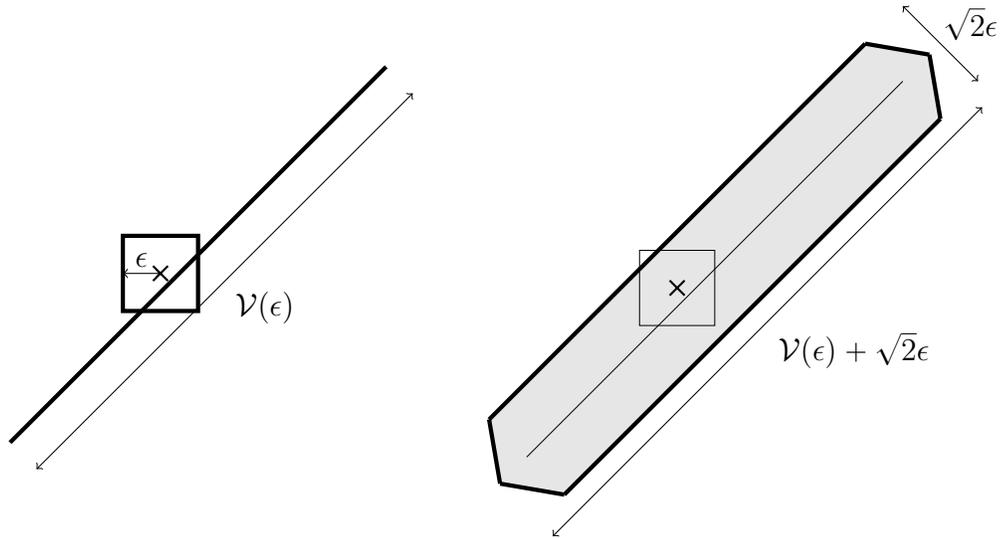

Upon rescaling the set $\mathfrak{F}$ and multiplying the visibility function by a constant depending only on the dimension, the construction of a dense forest with visibility $\mathcal{V}$ amounts to solving this question~: 

\begin{pb}[Construction of Dense Forests with a given Visibility Bound]\label{pbgoodviibility} Let $\mathcal{V}\; :\; \epsilon>0\mapsto\mathcal{V}\left(\epsilon\right)>0$ be such that $\mathcal{V}\left(\epsilon\right)\ge \epsilon^{-(n-1)}$ for all $\epsilon>0$ small enough. Does there exist a set of points $\mathfrak{F}\subset\R^n$ with finite density such that for any transformation $\tilde{g}$ that can be decomposed as
\begin{equation}\label{decomg'}
\tilde{g}=tk\tilde{a},\qquad \textrm{where} \qquad \tilde{a}\;=\;\textrm{\emph{Diag}}(\underbrace{\epsilon,\dots,\epsilon}_{n-1}, \mathcal{V}\left(\epsilon\right))\in\R^{n\times n},
\end{equation} 
where $k$ is an element of the special orthogonal group and where $t$ is a translation, the relation $$g'\left([0,1]^n\right)\cap\mathfrak{F}\;\neq\;\emptyset$$ holds for any $\epsilon>0$ small enough?
\end{pb}

Note that in the case that $\mathcal{V}\left(\epsilon\right)=O\left(\epsilon^{-(n-1)}\right)$ , Problem~\ref{pbgoodviibility} is equivalent to Problem~\ref{mainpbdenseforest}.\\

Denote by $\tilde{G}_n\!\left(\mathcal{V}\right)$ the set of all those transformations $g'$ that can be decomposed as in~\eqref{decomg'} and set furthermore $$\tilde{G}_n=\tilde{G}_n\!\left(\epsilon\mapsto\epsilon^{-(n-1)}\right).$$ Note that $\tilde{G}_n$ is a subset of the affine group $G_n$ defined in \S~\ref{equistatesec}, which is furthermore contained in the set $G'_n\subset G_n$ defined in \S~\ref{togosec}. Comparing the decomposition of an element $g'\in G'_n$ given in~\eqref{decomp} and that of an element $\tilde{g}\in\tilde{G}_n$ as above leads to the following observation on the relation between Problem~\ref{mainpbdenseforest}  and the Danzer Problem~: the concept of a dense forest with optimal visibility requires the consideration of a one--parameter subgroup of $SL_n\left(\R\right)$ made of all those diagonal matrices $\tilde{a}$ with positive entries, whereas the Danzer problem requires the consideration of all such diagonal matrices (which is an $(n-1)$--parameter subgroup). This restricts the set of boxes under  consideration and  allows for the relaxation of the volume constraint in Danzer's Problem by requiring that the one-parameter family should not necessarily be contained in $SL_n\left(\R\right)$ (as is done when dealing with the set $\tilde{G}_n\!\left(\mathcal{V}\right)$).

A Danzer set is necessarily a dense forest with optimal visibility since  it always holds that $\tilde{G}_n\subset G'_n$. The converse, however, does not hold except in the case $n=2$ as one then has that $\tilde{G}_2 = G'_2$. In other words, Danzer's Problem and the construction of a dense forest with optimal visibility are equivalent in the planar case $n=2$ (only)\footnote{This corrects an erroneous claim made both in~\cite[\S 4 \& \S 7.2]{SW} and in~\cite[p.1]{SS}~: a dense forest with optimal visibility is not necessarily a Danzer set when $n\ge 3$ as the set of boxes to be considered in the latter case is bigger than the set of those boxes needed to establish the dense forest property. In particular, the argument provided in the footnote of~\cite[p.1]{SS} is incorrect~: an $n$ dimensional box with volume $s>0$ and shortest edge length $2\epsilon$ has its longest edge length  equal to at least $s/(2\epsilon)^{n-1}$, and not $(s/(2\epsilon))^{1/(n-1)}$. This invalidates the claimed  proof. The author thanks Ioannis Tsokanos for pointing out to him this mistake.}.

\subsection{The State of the Art}\label{stateartdenseforest} In order to answer the above--stated question raised by Bishop, Peres~\cite[\S 2.3]{Bishop} constructed the first known example of a dense forest. This construction, restricted to the plane, provides a visibility bound of  $O\left(\epsilon^{-4}\right)$. This bound has subsequently been improved to $O\left(\epsilon^{-3}\right)$  by Adiceam, Solomon and Weiss~\cite[\S5]{ASW} thanks to the introduction of Diophantine ideas detailed in \S\ref{udtsec} below. Peres' construction has so far remained the main source of examples for explicit constructions of dense forests with effective bounds on the visibility. In \S\ref{peressec}, we provide a far--reaching generalisation of it by relating the visibility therein to a problem of distribution modulo one. 

Following Peres' work, Solomon and Weiss~\cite[Theorem 1.3]{SW} showed the existence of a uniformly discrete dense forest in any dimension. Their construction is deterministic and is based  on the concept of complete unique ergodicity introduced in their paper. However, no visibility bound is provided. This issue was later addressed by Adiceam~\cite{Adiceam}~: adapting Peres' construction to higher dimensions and relying on ideas from Fourier Analysis and on sharp estimates for exponential sums, he showed the existence of a (non uniformly discrete) dense forest in any dimension $n\ge 2$  with visibility bound $O\left(\epsilon^{-2(n-1)-\eta}\right)$ for any $\eta>0$. The corresponding forest depends on the probabilitic choice of a parameter.

To date, the best known visibility bound is due to Alon~\cite{Alon}, and is obtained for a forest construction of a purely probabilistic nature relying on  the Lov\'asz Local Lemma from Probability Theory (see~\cite[Chap.5]{AS} for a statement of this lemma)~:

\begin{thm}[Alon, 2018]\label{forestalon}
There exists a uniformly discrete dense forest in the plane with visibility bound $\epsilon^{-1+o(1)}$. More precisely, this forest admits the visibility function $$\mathcal{V}\left(\epsilon\right)=2^{C\sqrt{\log\left(\epsilon^{-1}\right)}}\epsilon^{-1}$$ for some $C>0$.
\end{thm}

This visibility bound gets very close to the optimal $O\left(\epsilon^{-1}\right)$ from Problem~\ref{mainpbdenseforest} when $n=2$. In fact, it is claimed without proof at the end of~\cite{Alon}, on the one hand that the visibility estimate can further be improved to  $$\mathcal{V}\left(\epsilon\right)=O\left(\epsilon^{-1}\cdot\log\left(\epsilon^{-1}\right)\cdot\log\log\left(\epsilon^{-1}\right)\right),$$ and on the other that the construction can be generalised to any dimension $n\ge 3$ to provide a forest with similar visibility bound  of the form $O\left(\epsilon^{-(n-1)+o(1)}\right)$. The assumption of uniform discreteness relates Theorem~\ref{forestalon} to the Boshernitzan--Conway Problem~\ref{boshcon}~: it can be seen as a way to address this problem when relaxing the volume constraint. \\

Since Alon's construction is purely probabilistic in nature, one may formulate the following problem capturing all desirable properties of a dense forest~: 

\begin{pb}\label{pbeffecfor}
To construct a dense forest satisfying simultaneously these three pro\-per\-ties~: (1) the construction is deterministic; (2) the forest is uniformly discrete; (3) the forest admits a good visibility bound.
\end{pb}

To date, there is no known construction  satisfying these three conditions (for any visibility bound, however weak). A dense forest very nearly missing the deterministic requirement only (in some suitable sense) will nevertheless be described in \S~\ref{secgrid}. 

\subsection{In Search for Deterministic Constructions~: Peres' Forest and the Concept of Super--Uniformly Dense Sequences} \label{peressec} The goal in this section is to generalise Peres' above--mentioned construction. This will provide a class of dense forests which, although not uniformly discrete, will meet the other two conditions stated in Problem~\ref{pbeffecfor} assuming that  one is able to find a sequence satisfying some explicit distribution property modulo one.

\subsubsection{Peres' construction} We first start by briefly outlining Peres' idea from~\cite[\S 2.3]{Bishop} to address Bishop's forest question  (see also~\cite[\S 5]{ASW} for a related discussion). 

To begin with, assume that one is able to construct a planar forest, say $\mathfrak{F}_1$,  which is dense when considering only those line segments which are ``nearly horizontal'' in the sense that they are contained in lines with equation $$y\;=\; \alpha x +\beta \qquad \textrm{with} \qquad \left|\alpha\right|\le 1 \qquad \textrm{and}\qquad \beta\in\R.$$

Requiring that the visibility should be some function $\mathcal{V}$ thus means that for any $M\in\R$, any $\alpha\in [-1, 1]$, any  $\beta\in\R$ and any $\epsilon>0$, there exists $\left(a,b\right)\in\mathfrak{F}_1$ such that $$\max\left\{\left| M+x-a \right|, \left|\alpha (x+M)+\beta-b\right|\right\}\;<\;\epsilon\qquad\textrm{for some}\qquad 0\;\le\; x\;\le\; \frac{\mathcal{V}\left(\epsilon\right)}{\sqrt{1+\alpha^2}}\cdot$$ Since $\alpha^2\le 1$, this condition is fulfilled as soon as for any $M, \beta\in\R$, any $\alpha\in\R$ and any $\epsilon>0$,  there exists $\left(a,b\right)\in\mathfrak{F}_1$ such that  
\begin{equation}\label{condivisiperes}
\max\left\{\left| M+x-a \right|, \left|\alpha x+\beta-b\right|\right\}\;<\;\epsilon\qquad\textrm{for some}\qquad 0\;\le\; x\;\le\; \frac{\mathcal{V}\left(\epsilon\right)}{\sqrt{2}}\cdot
\end{equation}

Define at present $\mathfrak{F}_2$ as the point set obtained by rotating $\mathfrak{F}_1$ by an angle $\pi/2$. Clearly, $\mathfrak{F}_2$ will be a dense forest with visibility $\mathcal{V}$ for all those line segments  which are``nearly vertical'' in the sense that they are contained in lines with equation $$y\;=\; \alpha' x +\beta' \qquad \textrm{with} \qquad \left|\alpha'\right|\ge 1\qquad \textrm{and}\qquad \beta'\in\R.$$

The point set 
\begin{equation}\label{peresfoerest}
\mathfrak{F}\;=\; \mathfrak{F}_1\cup\mathfrak{F}_2
\end{equation} 
then becomes a dense forest with visibility $\mathcal{V}$. Peres' choice reduces to taking 
\begin{equation}\label{f_1peres}
\mathfrak{F}_1\;=\;\Z^2\cup \begin{pmatrix}  1 &
  0 \\ \varphi & 1\end{pmatrix}\cdot \Z^2 \;=\; \bigcup_{\left(k,l\right)\in\Z^2}\left\{\left(k, l\right), \left(k, \varphi k +l\right)\right\},
\end{equation}
where $\varphi= \left(1+\sqrt{5}\right)/2$ is the Golden Ratio. The resulting construction is represented in Figure~\ref{peresforest}.

\begin{figure}[h!]
\begin{center}
\includegraphics[scale=1]{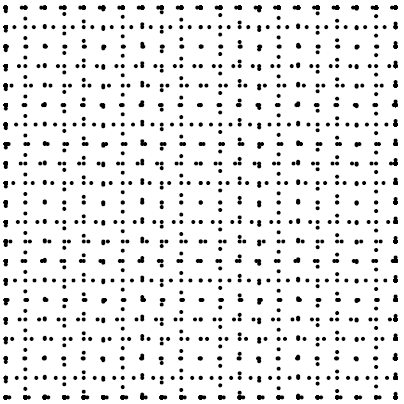}
\end{center}
\caption{Peres' dense forest $\mathfrak{F}$ is the union of two subsets~: the first one (corresponding to points lying on vertical lines) is a dense forest for those line segments with direction close to the horizontal; the second one (corresponding to points lying on horizontal lines) is a dense forest for those line segments with direction close to the vertical.}
\label{peresforest}
\end{figure}

Peres then obtains the visibility bound $\mathcal{V}\left(\epsilon\right)=O\left(\epsilon^{-4}\right)$ by showing that~\eqref{condivisiperes} holds for some $a=k\in\Z$ and some $x$ of the form $x=k-M$. This easily follows from Dirichlet's Theorem in Diophantine approximation (see~\cite[Lemma 2.4]{Bishop}). 

For the generalisation of this construction we have in mind, it must be noted that at the heart of the proof of the visibility bound $O\left(\epsilon^{-4}\right)$ lies the following  property~: the sequence $\left(b_k\right)_{k\ge 0}$ defined for all $k\ge 0$ by setting 
\begin{equation}\label{defbk}
b_{2k}=0 \qquad \textrm{ and }\qquad b_{2k+1}=\varphi k
\end{equation}
is such that for some constant $C>0$, for any integer $m\ge 0$, any $\xi\in\R$ and any $\epsilon>0$, the finite set of fractional parts $\left\{\left\{b_{k+m}-\xi(k+m)\right\}\; :\; 0 \le k\le C\cdot\epsilon^{-4}\right\}$ is $\epsilon$--dense in the unit interval $[0,1]$. This is saying that this finite set of points intersects any open interval with length $\epsilon$ contained in $[0,1]$.

\subsubsection{The Concept of Super--Uniform Density} In order to capture the generality behind Peres' construction, we consider sequences satisfying similar distribution properties in any dimension. To this end,  first recall that the \emph{dispersion} of a finite set of points $\bm{U}_N=\left\{ \bm{u}_0, \dots, \bm{u}_{N-1}\right\}$ lying in the unit hypercube $$I_n=[0, 1]^n$$ is defined as the quantity 
\begin{equation}\label{defdisp}
\delta\left(\bm{U}_N\right)\;=\; \max_{\bm{x}\in I_n}\; \min_{0\le  k\le N-1}\; \left\|\bm{u}_k-\bm{x}\right\|_{\infty}.
\end{equation}
It holds that an infinite sequence $\bm{U}=\left(\bm{u}_k\right)_{k\ge 0}$ in $I_n$ is dense in $I_n$ if, and only if, $$\lim_{N\rightarrow\infty} \delta\left(\bm{U}_N\right)\; =\; 0 $$ (see~\cite[Theorem 1.16]{DT} for a proof).

The concept of dispersion is not as easy to deal with as that of \emph{discrepancy}. Recall that the discrepancy of the above defined finite set of points $\bm{U}_N$ is the quantity
\begin{equation*}
\Delta\left(\bm{U}_N\right)\;=\; \sup_{J \subset I_n}\; \left|\frac{\#\left\{0\le k\le N-1\; :\; \bm{u}_k\in J\right\}}{N}-\lambda_n\left(J\right)\right|,
\end{equation*}
where the supremum is taken over all aligned boxes $J$ contained in $I_n$ and where $\lambda_n$ denotes the $n$--dimensional Lebesgue measure restricted to $I_n$. The sequence $\bm{U}$ is then equidistributed if, and only if,  $$\lim_{N\rightarrow\infty} \Delta\left(\bm{U}_N\right)\; =\; 0 $$ (see~\cite[Chap.1]{DT} for further details on the concept of equidistribution). \\

Dispersion and discrepancy are related as follows (see~\cite[p.12]{DT})~: for any $N\ge 0$,
\begin{equation}\label{reldispdiscr}
\frac{1}{2\left(N^{1/n}+1\right)}\;\le\;\delta\left(\bm{U}_N\right)\;\le\; \frac{\sqrt[n]{\Delta\left(\bm{U}_N\right)}}{2}\cdotp
\end{equation}
This relation implies in particular the natural claim that an equidistributed sequence is dense. However, these two concepts are rather different in nature~: equidistribution depends on the choice of a measure whereas density is a property dependent on the choice of a distance (here induced by the sup norm).  Correspondingly, the discrepancy and the dispersion of a finite set of points tend to behave very differently~: the latter can only decrease by the adjunction of a new point whereas the behaviour of the former depends very much on the position of this new point  with respect to the previous ones. Therefore, inequality~\eqref{reldispdiscr} should be seen as a crude estimate for the dispersion. It is nevertheless very useful inasmuch as a number of tools are available to estimate the discrepancy of a given sequence, e.g., the analytic methods resulting from Fourier Analysis and the Theory of Exponential Sums through Weyl's Equidistribution Criterion~\cite[Theorem 1.19]{DT} and the Erd\H{o}s--Tur\'an Inequality~\cite[Theorem 1.21]{DT}. Such an analytic approach is at the heart of the construction in~\cite{Adiceam} of a dense forest with visibility $O\left(\epsilon^{-2(n-1)-\eta}\right)$ for any $\eta>0$ in dimension $n\ge 2$.\\

To generalise  Peres' construction, we introduce two concepts. In the definition below, given $\bm{x}\in\R^n$, we denote by $\left\{\bm{x}\right\}$ the $n$--dimensional vector all of whose components are the fractional parts of the corresponding components of the vector $\bm{x}$.

\begin{defi}\label{defsud}
Let $\bm{V}=\left(\bm{v}_k\right)_{k\ge 0}$ be a sequence in $\R^n$. 

The sequence $\bm{V}$ is \emph{super uniformly dense modulo one} if, for any integer $m\ge 0$ and any $\bm{\xi}\in\R^n$, the sequence of fractional parts 
\begin{equation}\label{deffracpartsuxim}
\bm{U}\!\left(m, \bm{\xi}\right)=\left(\left\{\bm{v}_{k+m}-\bm{\xi}(k+m)\right\}\right)_{k\ge 0}
\end{equation}
 is dense in $I_n$ uniformly in $m$ and $\bm{\xi}$. Denoting by  $\bm{U}_N\!\left(m, \bm{\xi}\right)$ the first $N$ terms of the sequence $\bm{U}\!\left(m, \bm{\xi}\right)$, this is saying that 
 \begin{equation}\label{defdeltahat}
 \hat{\delta}_{\bm{V}}(N)\: :=\:\sup_{m\ge 0}\;\sup_{\bm{\xi}\in\R^n}\;\delta\left(\bm{U}_N\!\left(m, \bm{\xi}\right)\right)\;\longrightarrow\; 0 \qquad \textrm{as}\qquad N\rightarrow\infty.
 \end{equation}

The sequence $\bm{V}$ is \emph{super uniformly equidistributed modulo one} if, for any integer $m\ge 0$ and any $\bm{\xi}\in\R^n$, the sequence of fractional parts~\eqref{deffracpartsuxim} is equidistributed in $I_n$ uniformly in $m$ and $\bm{\xi}$. This is saying that 
$$\hat{\Delta}_{\bm{V}}(N)\: :=\:\sup_{m\ge 0}\;\sup_{\bm{\xi}\in\R^n}\;\Delta\left(\bm{U}_N\!\left(m, \bm{\xi}\right)\right)\;\longrightarrow\; 0 \qquad \textrm{as}\qquad N\rightarrow\infty.$$
\end{defi}

From~\eqref{reldispdiscr}, a super uniformly equidistributed sequence modulo one is also super uniformly dense. The latter concept seems to be new whereas the former encapsulates two well-studied ones. 

On the one hand,  super uniform equidistribution implies that the sequence $\bm{V}$ under consideration has, after reduction modulo one, an \emph{empty spectrum}. In the notation of the above definition, this means that for all $\bm{\xi}\in\R^n$, $$\Delta\left(\bm{U}_N\!\left(0, \bm{\xi}\right)\right)\;\longrightarrow\; 0 \qquad \textrm{as}\qquad N\rightarrow\infty;$$ that is, that for all $\bm{\xi}\in\R^n$, the sequence of fractional parts $\left(\left\{\bm{v}_{k}-\bm{\xi}k\right\}\right)_{k\ge 0}$ is equidistributed (if this sequence is not equidistibuted for some $\bm{\xi}\in\R^n$, the vector $\bm{\xi}$ is said to lie in the spectrum of the sequence $\bm{V}$). The spectral property of sequences were extensively studied in the 1970's, in particular from the point of view of numeration systems. The reader is referred to the bibliography mentioned in~\cite[p.2068]{BBLT} and to~\cite[\S 1.4.3]{DT} for further details on the topic. Here, we content ourselves with indicating that the spectrum of \emph{any} sequence has always zero Hausdorff dimension, and that almost all sequences in the interval $[0,1)$ have empty spectrum (with respect to the Haar measure on the infinite torus $\left(\R\backslash\Z\right)^{\N}$ after identifying $[0,1)$ with  $\R\backslash\Z$ in the natural way). See~\cite[\S 4]{DMF} for a proof of these claims.

On the other hand, super uniform equidistribution implies that the sequence $\bm{V}$ is \emph{well--distributed} after reduction modulo one. This means that $$\sup_{m\ge 0}\;\Delta\left(\bm{U}_N\!\left(m, \bm{0}\right)\right)\;\longrightarrow\; 0 \qquad \textrm{as}\qquad N\rightarrow\infty;$$ that is, that the sequence $\left(\left\{\bm{v}_{k+m}\right\}\right)_{k\ge 0}$ is equidistributed uniformly in $m\ge 0$. This property is not generic anymore~: indeed, it follows from~\cite[Theorem 3.8]{KN} for instance that almost no sequence in the interval $[0,1]$ is well--distributed. The interested reader is also referred to the nice survey~\cite[\S 3]{BM} for further properties of well--distributed sequences.\\

The motivation behind the above definition is the following theorem, the proof of which will be given in \S\ref{proofperes}~:

\begin{thm}[Visibility in Generalised Peres' Forests]\label{uniquethm} Assume that the sequence $\bm{V}=\left(\bm{v}_k \right)_{k\ge 0}$ taking values in $\R^{n-1}$ is super uniformly dense. Then, there exists a dense forest in $\R^n$ with visibility 
\begin{equation}\label{bornevsiiperes}
\mathcal{V}\left(\epsilon\right)\;=\; O\left(f_{\bm{V}}\left(\epsilon \right)\right), \qquad \textrm{where} \qquad f_{\bm{V}}\left(\epsilon \right)=\min\left\{N\ge 0\; :\; \hat{\delta}_{\bm{V}}\left(N\right)<\epsilon\right\}
\end{equation} 
and where $\hat{\delta}_{\bm{V}}$ is the measure of super uniform dispersion defined in~\eqref{defdeltahat}. 

Furthermore, the construction of such a forest is explicit provided that  the sequence $\bm{V}$ is deterministic.
\end{thm}

Note that as the dispersion function $\hat{\delta}_{\bm{V}}$ gets closer to its best possible lower bound given by the left--hand side inequality in~\eqref{reldispdiscr} (the exponent to be considered in this context is $1/(n-1)$), the visibility function in~\eqref{bornevsiiperes} gets closer to the optimal bound stated in Problem~\ref{mainpbdenseforest}.\\

As an application of Theorem~\ref{uniquethm}, let $\bm{V}$ be the one-dimensional sequence $\left(\alpha k^2\right)_{k\ge 0}$, where $\alpha$ is a badly approximable number (this means that $\left|q\right|\cdot\left|q\alpha -p\right|>c$ for some $c>0$ and all integer vectors $\left(p,q\right)\neq\left(0,0\right)$). It is an elementary exercise, combining the Erd\"os--Tur\'an Inequality~\cite[Theorem 1.21]{DT} and the known sharp estimates for quadratic exponential sums~\cite[Theorem~6]{FJK}, to show that, in this case, regardless of the choice of $m\ge 0$ and of $\xi\in \R$, it holds that $$\Delta\left(\bm{U}_N\!\left(m, \bm{\xi}\right)\right)\; =\; O\left(\frac{1}{\sqrt[3]{N}}\right)$$  (the implicit constant depends only on $\alpha$). Theorem~\ref{uniquethm} and the right--hand side ine\-qua\-li\-ty in~\eqref{reldispdiscr} (where $n=1$) then imply the existence of an explicit planar  dense forest with visibility $O\left(\epsilon^{-3}\right)$.\\

The following problem naturally follows from the statement of Theorem~\ref{uniquethm}~:

\begin{pb} \label{pbsudgen}
Does there exist a sequence $\bm{V}=\left(\bm{v}_k\right)_{k\ge 0}$ in $\R^d$ ($d\ge 1$) such that its measure of super uniform dispersion $\hat{\delta}_{\bm{V}}$ satisfies the relation $\hat{\delta}_{\bm{V}}\left(N\right)=O\left(N^{-1/d}\right)$? More generally, what is the best decay rate that can be obtained for the quantity $\hat{\delta}_{\bm{V}}$ as $N$ tends to infinity?
\end{pb}

Theorem~\ref{udtborelcantelli} in \S\ref{udtsec} implies the existence of a probabilistic sequence $\bm{V}$ in $\R^d$ achieving the bound $\hat{\delta}_{\bm{V}}\left(N\right)=O\left(N^{-1/d+\eta}\right)$ for any $\eta>0$. In order to construct deterministic dense forests with good visibility bounds, the following pro\-blem is also of interest~:

\begin{pb}\label{pbexplicitsud}
What is the best decay rate that can be obtained for the measure of super uniform dispersion $\hat{\delta}_{\bm{V}}$ assuming that the $d$--dimensional sequence $\bm{V}=\left(\bm{v}_k\right)_{k\ge 0}$ is deterministic?
\end{pb}

\sloppy The best known result in this direction in the case $d=1$ is due to Tsokanos~\cite{T} and results from the construction of a dyadic sequence~:

\begin{thm}[Tsokanos, 2020+]\label{tsok}
There exists a deterministic super uniformly dense sequence $\bm{V}$ in $\R$ such that for any $\eta>0$, $\hat{\delta}_{\bm{V}}(N)=O\left(N^{-1/2+\eta}\right)$. As a consequence, one can construct an explicit dense forest in the plane with visibility $O\left(\epsilon^{-2-\eta}\right)$ for any $\eta>0$.
\end{thm}

The sequence $\bm{V}=\left(v_n\right)_{n\ge 1}$ in this statement is defined as follows~: decompose the integer $n\ge 1$ as $n=k\cdot 2^i+2^{i-1}-2$ with $i\ge 1$ and $k\ge 0$, and the integer $k$ as $k=r\cdot 2^{i^2+2}+s$ with $0\le r\le 2^{i^2+2}-1$ and $1\le s\le 2^{i^2+2}$. Then, $$v_n\;=\; \begin{cases}
\frac{rs}{2^{(2i^2+4)}} & \textrm{if } r \textrm{ is odd,} \\
\frac{rs}{2^{(2i^2+4)}} + \frac{s}{2^{(i^2+4)}} & \textrm{if } r \textrm{ is even}.
\end{cases}$$

The resulting forest visibility bound is actually slightly sharper than the one appearing in Theorem~\ref{tsok} as it is of the form $O\left(\epsilon^{-2}\cdot 2^{C\sqrt{-\log(\epsilon)}}\right)$ for some $C>0$. 

\subsubsection{Dense Forests Constructed from Super Uniformly Dense Sequences.}\label{proofperes} The goal of this section is to prove Theorem~\ref{uniquethm}. The construction of the dense forest described here generalises ideas appearing in~\cite{Adiceam}, \cite[\S 5]{Adiceambis} and~\cite[\S 2.3]{Bishop}. The main substance of the proof is contained in this lemma~:

\begin{lem}\label{lemf1}
Let $\bm{V}=\left(\bm{v}_k\right)_{k\ge 0}$ be a sequence in $\R^{n-1}$ satisfying the assumptions of Theorem~\ref{uniquethm}. Extend it to a sequence indexed by $\Z$ by setting $\bm{v}_k=\bm{v}_{-k}$ if $k$ is a strictly negative integer. Define furthermore $$\mathfrak{F}_1\; =\; \bigcup_{\left(k, \bm{l}\right)\in\Z\times\Z^{n-1}}\left\{\left(k, \bm{v}_k+ \bm{l}\right)\right\}\;\subset\; \R^n.$$ 
Then, $\mathfrak{F}_1$ is a dense forest with visibility $O\left( f_{\bm{V}}\left(\epsilon\right)\right)$ for those line segments contained in lines ``close to the $x_1$--axis''; more precisely, for those line segments contained in lines of the form 
\begin{equation}\label{linerstric}
\left\{\bm{\alpha}t+\bm{\beta}\right\}_{t\in\R},\quad \textrm{where}\quad  \bm{\beta}\in\R^n\quad \textrm{and}\quad \bm{\alpha}=\left(\alpha_1, \dots, \alpha_n\right)\in\mathbb{S}^{n-1}\quad  \textrm{with}\quad \left\|\bm{\alpha}\right\|_{\infty}=\left|\alpha_1\right|.
\end{equation}
\end{lem}

\begin{proof}
It is easy to check that $\mathfrak{F}_1$ is uniformly discrete. With the notation introduced in Theorem~\ref{uniquethm}, the goal is thus to show the existence of a constant $C>0$ such that for any $\epsilon>0$ and any $t_0\in\R$, there should exist
\begin{equation}\label{condforet}
\tau\in \left[t_0,\;  t_0+C\cdot f_{\bm{V}}\left(\epsilon\right)\right] \quad\textrm{for which} \quad \left\|\bm{\alpha}\tau+\bm{\beta}-\bm{f}_1\right\|_{\infty}<\epsilon \quad\textrm{for some}\quad \bm{f}_1\in\mathfrak{F}_1.
\end{equation} 
To this end, restrict $\tau$ to be of the form $\tau=(p-\beta_1)/\alpha_1$, where $p$ is an integer and $\bm{\beta}=\left(\beta_1, \bm{\beta}'\right)$ with $\bm{\beta}'\in\R^{n-1}$. Decomposing the vector $\bm{\alpha}\in\R^n$ as $\bm{\alpha}=\left(\alpha_1, \bm{\alpha}'\right)$ with $\bm{\alpha}'=\left(\alpha_2, \dots, \alpha_n\right)\in\R^{n-1}$ and setting $\bm{f}_1=\left(p, \bm{v}_p+\bm{l}\right)$ for some $\bm{l}\in\Z^{n-1}$, condition~\eqref{condforet} becomes $$\left\|p\bm{\xi}+\bm{\delta}-\bm{v}_k+\bm{l}\right\|_{\infty}<\epsilon,$$ where $\bm{\xi}=\bm{\alpha}'/\alpha_1$ and $\bm{\delta}=\bm{\beta}'-\left(\beta_1/\alpha_1\right)\cdot\bm{\alpha}'$. Here, the first requirement in~\eqref{condforet} translates into $$\beta_1+\alpha_1t_0<p<\beta_1+\alpha_1t_0+\alpha_1C\cdot  f_{\bm{V}}\left(\epsilon\right)\qquad \textrm{if}\qquad \alpha_1>0$$ and $$\beta_1+\alpha_1t_0>p>\beta_1+\alpha_1t_0+\alpha_1C\cdot  f_{\bm{V}}\left(\epsilon\right)\qquad \textrm{if}\qquad \alpha_1<0.$$

Under the restriction imposed in~\eqref{linerstric}, it holds that $1/\sqrt{n-1}\le\left|\alpha_1\right|\le 1$. Set $p=k+m$ with $m$ the least integer bigger than $\beta_1+\alpha_1t_0$ if $\alpha_1>0$ and $m$ the least integer bigger than $\beta_1+\alpha_1t_0+\alpha_1C\cdot  f_{\bm{V}}\left(\epsilon\right)$ if $\alpha_1<0$. It is then enough to establish the existence of a constant $C>0$ such that  for any $\epsilon>0$, any $m\in\Z$ and any $\bm{\xi}, \bm{\delta}\in\R^{n-1}$ , it holds that 
\begin{equation}\label{presqconclu}
\left\|\left\{\bm{v}_{k+m}+\bm{\xi}(k+m)-\bm{\delta}\right\}\right\|_{\infty}<\epsilon\quad\textrm{for some}\quad 1\le k\le C\cdot f_{\bm{V}}\left(\epsilon\right).
\end{equation}
Here, at least $ C\cdot f_{\bm{V}}\left(\epsilon\right)/2$ of the integers of the form $k+m$ as above are either non--negative or  else strictly negative. In the former case, it is enough to show that~\eqref{presqconclu} holds for some $k$ such that $0\le m+(C\cdot f_{\bm{V}}+1)/2\le k\le C\cdot f_{\bm{V}}$. In the latter case, since  $\bm{v}_{-k}=\bm{v}_k$ for all $k\ge 0$, the same reduction can be made. This shows that it suffices to establish that~\eqref{presqconclu} holds under the additonal assumption that $k\ge 0$ and that $m\ge 0$ (upon redefining the integer $m$ as $\pm\left(m+(C\cdot f_{\bm{V}}+1)/2\right)$ and the value of the constant $C$ as $C/2$).

Recalling the definition of dispersion in~\eqref{defdisp}, assumption~\eqref{defdeltahat} implies the existence of a smallest integer $N=N\!\left(\epsilon\right)\ge 0$ depending only on $\epsilon$ such that~\eqref{presqconclu} is valid uniformly in the choice of the integer $m\ge 0$ and of the vectors $\bm{\xi}, \bm{\delta}\in\R^{n-1}$. This integer $N$ is by definition equal to $ f_{\bm{V}}\left(\epsilon\right)$. Upon ensuring that  $C\ge 1$, the proof is therefore complete.
\end{proof}

In order to conclude  the proof of Theorem~\ref{uniquethm},  define the set $\mathfrak{F}$ as the union 
\begin{equation}\label{forestunion}
\mathfrak{F}\;=\; \bigcup_{j=1}^{n}\mathfrak{F}_j.
\end{equation} 
Here, $\mathfrak{F}_j$ is obtained from the set $\mathfrak{F}_1$ defined in Lemma~\ref{lemf1} by applying to it the rotation $\mathcal{R}_j$ which brings the $x_1$--axis onto the $x_j$--axis in the $(x_1, x_j)$--plane and leaves the orthogonal of this plane invariant (in particular, $\mathcal{R}_1$ is the identity). 

\begin{proof}[Completion of the proof of Theorem~\ref{uniquethm}] The set $\mathfrak{F}$ has finite density as a finite union of sets with finite density\footnote{In fact, it is not hard to see that $\mathfrak{F}$ has  finite upper uniform density in the sense that~\eqref{defupperfinitedensity} holds.}. It is furthermore a dense forest with visibility $O\left( f_{\bm{V}}\left(\epsilon\right)\right)$. Indeed, let $\Lambda$ be a line  segment parametrised as 
\begin{equation*}
\Lambda\;=\; \left\{\bm{\alpha}t+\bm{\beta}\;:\; t_0<t<t_0+C\cdot f_{\bm{V}}\left(\epsilon\right)\right\}
\end{equation*}
for some $t_0\in\R$ and some $C>0$. Here, $\bm{\beta}\in\R^n$ and 
\begin{equation*}
\bm{\alpha}=\left(\alpha_1, \dots, \alpha_n\right)\in\mathbb{S}^{n-1}\quad  \textrm{with}\quad \left\|\bm{\alpha}\right\|_{\infty}=\left|\alpha_j\right|
\end{equation*}
for some $1\le j\le n$. Then,  the  line  segment $\mathcal{R}^{-1}_j\left(\Lambda\right)$ is contained in a line of the form~\eqref{linerstric}. For some suitable value of $C>0$, Lemma~\ref{lemf1} thus implies that it gets $\epsilon$--close to a point in $\mathfrak{F}_1$. This is precisely saying that, for this value of $C>0$,  $\Lambda$ gets $\epsilon$--close to a point in $\mathfrak{F}_j$. 
This concludes the proof of Theorem~\ref{uniquethm}.
\end{proof}

\subsection{On Uniformly Diophantine Sets of Vectors}\label{udtsec} The best known result towards Problem~\ref{pbsudgen} is due to Adiceam, Solomon and Weiss, although the correspon\-ding theorem, namely~\cite[Theorem 5.3]{ASW}, is not stated in the language of super uniform density.  It relies on the newly introduced concept of sets of vectors which are Uniformly Diophantine and provides, as detailed below, a probabilistic sequence $\bm{V}$ in dimension $\R^{n-1}$ with super uniform dispersion $\hat{\delta}_{\bm{V}}\left(N\right)=O\left(N^{-1/{(n-1)}+\eta}\right)$ for any $\eta>0$. The purpose of this section is to shed light on this new concept by showing how it fits in a more general circle of ideas that can be used to define sequences which are super uniformly dense.\\

Cover the unit cube $I_d$ into $s\ge 1$ subsets, say $\mathcal{C}_1, \dots, \mathcal{C}_s$. One thus has the (non--necessarily disjoint) union 
\begin{equation}\label{partition}
[0,1]^d\;=\; \bigcup_{i=1}^s\mathcal{C}_i.
 \end{equation}
 Assume that for each set $\mathcal{C}_i$ ($1\le i\le s)$, one can find a sequence $\bm{A}^{(i)}=\left(\bm{a}^{(i)}_k\right)_{k\ge 1}$ in $\R^d$ which is super uniformy dense when the spectral part is restricted to $\mathcal{C}_i$; that is, such that, in the notation of Definition~\ref{defsud},  
 \begin{equation}\label{condnew}
 \sup_{m\ge 0}\;\sup_{\bm{\xi}\in\mathcal{C}_i}\;\delta\left(\bm{A}^{(i)}_N\!\left(m, \bm{\xi}\right)\right)\;\longrightarrow\; 0 \qquad \textrm{as}\qquad N\rightarrow\infty.
 \end{equation}

Let then $\bm{B}=\left(\bm{b}_k\right)_{k\ge 1}$ be the $d$--dimensional sequence obtained by concatenating the sequences $\bm{A}^{(i)}$ ($1\le i\le s$) in the following way~: given $k\ge 1$, set 
\begin{equation}\label{seqB}
\bm{b}_k\;=\; \bm{a}^{(i)}_{js+i}\qquad \textrm{if}\qquad k=js+i;  
\end{equation} 
that is, $$\left(\bm{b}_k\right)_{k\ge 1}\;=\;\left(\bm{a}^{(1)}_{1}, \bm{a}^{(2)}_{1}, \dots   , \bm{a}^{(s)}_{1},\;\; \bm{a}^{(1)}_{2}, \bm{a}^{(2)}_{2}, \; \dots\; , \bm{a}^{(s)}_{2}, \; \dots\right).$$
Since the terms of a given  subsequence  $\bm{A}^{(i)}$ have bounded gaps within the sequence $\bm{B}$, it is easy to check that $\bm{B}$ is super  uniformly dense; that is, that
 \begin{equation*}
\hat{\delta}_{\bm{B}}(N)\;=\; \max_{1\le i\le s}\; \sup_{m\ge 0}\;\sup_{\bm{\xi}\in\mathcal{C}_i}\;\delta\left(\bm{A}^{(i)}_N\!\left(m, \bm{\xi}\right)\right)\;\longrightarrow\; 0 \qquad \textrm{as}\qquad N\rightarrow\infty.
 \end{equation*}

This strategy is implicitly implemented in~\cite[\S 5]{ASW} by choosing $\bm{\theta}_1, \dots, \bm{\theta}_s\in\R^d$ such that~\eqref{partition} holds with 
\begin{equation}\label{defCi}
\mathcal{C}_i\;=\;\left\{\bm{\xi}\in [0,1]^d\; :\; \textrm{the finite sequence }\left(\left\{\left(\bm{\theta}_i-\bm{\xi}\right)\cdot k\right\}\right)_{1\le k\le \mathcal{V}\left(\epsilon\right)} \textrm{ is }\epsilon\textrm{--dense in }I_d\right\}.
\end{equation}
The sequence $\bm{B}$ in~\eqref{seqB} is then defined by setting 
\begin{equation}\label{defai}
\bm{a}^{(i)}_k\; =\; \frac{\bm{\theta}_i}{s}\cdot k
\end{equation} 
for all $k\ge 1$. (The normalising factor $1/s$ just comes from the fact that when seeing $\bm{A}^{(i)}$ as a subsequence of $\bm{B}$, definition~\eqref{seqB} implies that the specific integers $k$ to be considered will contain an additional multiplicative factor  $s$.)  In~\eqref{defCi}, $\mathcal{V}$ is a function to be determined for the sets $\mathcal{C}_i$ ($1\le i\le s$) to be well--defined.  It is not hard to see that the existence of such a  function guarantees that~\eqref{condnew} holds, and furthermore that, in the notation of Theorem~\ref{uniquethm},
\begin{equation}\label{bornvisifv}
f_{\bm{B}}(\epsilon)\;\le \; s\cdot\mathcal{V}\!\left(\epsilon\right).
 \end{equation}

Note that when $d=1$, $\theta_1=0$ and $\theta_2=\varphi$ (the Golden Ratio), the sequence $\bm{B}$ thus obtained  reduces precisely to the sequence~\eqref{defbk} appearing in Peres' planar construction. The point to consider linear sequences is that they behave well under the two conditions imposed by the property of super uniform density~: on the one hand, adding a linear shift (as is required when considering the spectrum) preserves the linearity property; on the other, and most importantly, the dispersion of a linear sequence does not depend on the starting index (in the sense that, if $\bm{A}=\left(\bm{a}_k\right)_{k\ge 0}$ is linear, then, with the notation introduced in~\eqref{defdisp} and~\eqref{deffracpartsuxim} , $\delta\left(\bm{A}_N\right)=\delta\left(\bm{A}_N(m, \bm{0}\right)$ for all $m, N\ge 0$).\\

In order to show that the sets $\mathcal{C}_i$'s defined in~\eqref{defCi} cover the interval $I_d$ for some function $\mathcal{V}$ and some vectors $\bm{\theta}_1, \dots, \bm{\theta}_s\in\R^d$ to be determined, the following definition is introduced in~\cite[\S 5]{ASW}. Therein, the scalar product between two vectors $\bm{x}, \bm{y}\in\R^d$ is denoted by $\bm{x\cdot y}$ and $$\left\|\bm{x}\right\|_{\Z^d}\; :=\;\min_{\bm{l}\in\Z^d}\;\left\|\bm{x}-\bm{l}\right\|_{\infty}.$$

\begin{defi}\label{defudt}
Let $\Phi$ be a non--increasing function such that $\Phi(T)\rightarrow 0$ as $T\rightarrow\infty$. A set of $s$ vectors $\bm{\theta}_1, \dots, \bm{\theta}_s$ in $\R^d$ is \emph{Uniformly Diophantine with type} $\Phi$ if   the following holds~: for all $\bm{\xi}\in\R^d$ and all $T\ge 1$, there exists an index $i\in\llbracket 1, s \rrbracket$ such that for all non-zero  $\bm{u}\in\Z^d$ with sup norm at most $T$, one has 
\begin{equation}\label{inedudt}
\left\|\bm{u\cdot}\left(\bm{\xi}-\bm{\theta}_i\right)\right\|_{\Z^d}\;\ge\; \Phi(T).
\end{equation}
\end{defi}

The following is established in the proof of~\cite[Theorem 5.2]{ASW}~: if $\bm{\theta}_1, \dots, \bm{\theta}_s\in\R^d$ is a set of $s$ uniformly Diophantine vectors with type $\Phi$, then~\eqref{partition}  holds provided that 
\begin{equation}\label{defvisiudt}
\mathcal{V}\left(\epsilon\right)\;=\; O\left(\left(\epsilon^{d-1}\cdot\Phi\left(d\epsilon^{-1}\right)^{-1}\right)^d\right).
\end{equation}
 From the above discussion, the sequence $\bm{B}$ obtained from~\eqref{defai} is then super uniformly dense, and inequality~\eqref{bornvisifv} shows that  the corresponding generalised Peres' forest constructed in the proof of Theorem~\ref{uniquethm}  has visibility~\eqref{defvisiudt}. Upon setting 
 \begin{equation}\label{dimrel}
 n=d+1
\end{equation}
and representing vectors as columns, this forest $\mathfrak{F}\subset\R^n$ can be explicitly described as follows~: it is the union defined in~\eqref{forestunion}, where 
\begin{equation}\label{f_1peresgen}
\mathfrak{F}_1\;=\;\bigcup_{i=1}^{s} \begin{pmatrix}  1 &
  \bm{0}^T \\ \bm{\theta}_i & \textrm{Id}_d\end{pmatrix}\cdot \Z^n
  \end{equation} 
with $\textrm{Id}_d$ the $d\times d$ identity matrix (compare with the special case~\eqref{f_1peres}).
  
It should be noted here that, from Dirichlet's Theorem in Diophantine Approximation (see~\cite[Theorem VI]{Casselsbis}), the left--hand side in~\eqref{inedudt} is always less than $T^{-d}$ for some non zero integer vector $\bm{u}$ with sup norm at most $T$. Therefore, $$\Phi(T)\le T^{-d},$$ and the closer the function $\Phi$ gets to this upper bound, the closer the visibility estimate~\eqref{defvisiudt} gets to the optimal bound $O\left(\epsilon^{-(n-1)}\right)$ stated in Problem~\ref{mainpbdenseforest} (recall here~\eqref{dimrel}). \\
  
A sufficient condition, based on a reformulation of Definition~\ref{defudt} in terms of ap\-pro\-xi\-mation of a point by rational subspaces, is provided in~\cite[Proposition 7.1]{ASW} to establish that a set of vectors  is \emph{not} uniformly Diophantine of a given type. In the case of two numbers $\theta_1, \theta_2$ (that is, when $s=2$ and $n=d+1=2$), this condition can be refuted for a suitable type $\Phi$ if the difference $\theta_2-\theta_1$ is badly approximable. Setting $\theta_1=0$ and $\theta_2=\varphi$, this leads one to the improvement for the visibility in Peres' planar forest mentioned in \S\ref{stateartdenseforest}, namely from $O\left(\epsilon^{-4}\right)$ to $O\left(\epsilon^{-3}\right)$ (see~\cite[\S 7.1]{ASW} for details).

In the general case however, this sufficient condition does not reduce to any known concept in Diophantine Approximation. It is used in~\cite[Theorem]{ASW} to establish the existence of Uniformly Diophantine sets of vectors of a given type in full generality~: 

\begin{thm}[Adiceam, Solomon \& Weiss, 2020+]\label{udtborelcantelli}
Let $\Phi~: T\mapsto \Phi(T)$ be non--increasing and tending to zero at infinity. Assume that $s\ge d+1$ and that  
\begin{equation}\label{series}
\sum_{m=1}^{\infty}2^{md(s+1)}\Phi\left(2^m\right)^{s-d}\;<\;\infty.
\end{equation}
Then, the $s$ columns of almost all $d\times s$ real matrices determine a set of Uniformly Diophantine $d$--dimensional vectors with type $O\left(\Phi\right)$.
\end{thm}

The measure implicit in this a.e.~statement is the $(d\times s)$--dimensional Lebesgue measure. Note that the convergence assumption~\eqref{series} holds when setting, for a given $\eta''>0$, $\Phi(T)=T^{-\alpha}$, where $\alpha=d(s+1)/(s-d)+\eta''$ and $s\ge d+1$. From~\eqref{defvisiudt}, one deduces that the visibility in the resulting forest is $\mathcal{V}\left(\epsilon\right)=O\left(\epsilon^{-\beta}\right)$, where $\beta=d(s+d^2)/(s-d)+d\eta''$. This exponent $\beta$ is, for $s$ large enough and upon choosing the parameter $\eta''$ suitably, less than $d+\eta'=n-1+\eta'$ for any predefined $\eta'>0$. It is then easy to deduce from~\eqref{bornvisifv} that the sequence $\bm{B}$ defined by~\eqref{seqB} and~\eqref{defai} is such that $\hat{\delta}_{\bm{B}}(N)\;=\; O\left(N^{-1/(n-1)+\eta}\right)$ for any predefined $\eta>0$ for a suitably chosen $\eta'$.\\ 

In view of the statement of Theorem~\ref{udtborelcantelli}, a number of open problems emerge. They also appear in~\cite[\S 8]{ASW}.

\begin{pb}\label{pbemptiness}
Let $d\ge 1$ and $s\ge d+1$ be integers. What is the best decay rate that the function $\Phi$ can attain for there to exist a set of $s$ Uniformly Diophantine vectors with type $\Phi$ in dimension $d$?
\end{pb}

Clearly, Problem~\ref{pbemptiness} has direct implications towards Problems~\ref{mainpbdenseforest} and~\ref{pbsudgen}. As far as the question of explicit construction posed in Problem~\ref{pbexplicitsud} is concerned, one may ask~:

\begin{pb}\label{pbeffecexplisud}
Assume that $d\ge 1$ and $s\ge d+1$. Let $\Phi$ be a decaying map tending to zero at infinity such that the set of $s$-tuples of $d$--dimensional Uniformly Diophantine vectors with type $\Phi$ is non--empty. Can one find an explicit element in this set?
\end{pb}

The concept of uniformly Diophantine sets of vectors can be seen as being related to that of bad approximability in the case that the rational approximations under consideration are not defined by linear forms but by multilinear forms (this is, at the end of the day, the main idea formalised  in~\cite[Proposition 7.1]{ASW} to provide a  sufficient condition for a set of vectors not to be of a given uniformly Diophantine type). In this respect, Problem~\ref{pbeffecexplisud} must be linked to the fact  that explicit examples of badly approximable vectors, constructed from sets of algebraic conjugates, are known in any dimension --- see~\cite[\S II.4]{SchmidtDioph}.

The last problem in this section is related to the construction, in dimension 2, of a \emph{uniformly discrete} dense forest appearing  in~\cite[\S 6.1]{ASW} (this construction will further be discussed  in \S\ref{secgrid}). As it relies on the existence of sets of uniformly Diophantine vectors on manifolds, its generalisation to higher dimensions is likely to require first the solution to this problem, interesting on its own~:

\begin{pb}\label{udtmanifolds}
Assume that $d\ge 1$ and $s\ge d+1$. Let $\Phi$ be a decaying map tending to zero at infinity. Consider the set of $s$--tuples of $d$--dimensional Uniformly Diophantine vectors  with type $\Phi$. Does it intersect nondegenerate analytic manifolds non--trivially? 
\end{pb}

\section{The Danzer Property and Visibility in Structured Discrete Point Sets}\label{seccandidate} Given a set with finite density, natural questions emerging from Problems~\ref{danzer} and~\ref{mainpbdenseforest} are, on the one hand to determine if it is a dense forest and, on the other, whether it can be a Danzer set. Answering such questions is usually possible only if the point set under consideration enjoys some structural property. In this section,  we indicate what is known and less--known with respect to Problems~\ref{danzer} and~\ref{mainpbdenseforest} when considering a number of structured point sets with finite density; namely~: grids, cut--and--project sets, point sets obtained from a tiling resulting from a primitive substitution system, and the set of holonomy vectors of saddle connections on translation surfaces.

\subsection{Grids}\label{secgrid} Recall that a grid is a translated lattice and that, as seen in \S\ref{statartdanzer}, Bambah and Woods showed in the proof of ~\cite[Theorem 2]{BW} that a union of such sets can satisfy the Danzer Property while having a density growing like $O\left(T^n\left(\log T\right)^{n-1}\right)$ in dimension $n\ge 2$. Such union must necessarily be infinite~: indeed, as established in~\cite[Theorem 1]{BW}, a finite union of grids (which has necessarily finite density) cannot be a Danzer set.

On the other hand, Peres' planar construction  (see~\eqref{peresfoerest} and~\eqref{f_1peres}) and its generalisation by Adiceam, Solomon and Weiss (see~\eqref{forestunion} and~\eqref{f_1peresgen})  show that a  union of lattices can be a dense forest in any given dimension $n\ge 2$. 

Under the additional requirement that such a dense forest must be uniformly discrete, the following is known. A union of two grids in $\R^n$ ($n\ge 2$), when determining a uniformly discrete set, cannot be a dense forest in a strong sense~: as proved in~\cite[Corollary 2.2]{ASW}, it admits a \emph{vacant strip} in the sense that there exists a neighbourhood of a  hyperplane with constant width not containing any point in this union. However, it is established in~\cite[Theorem 1.2]{ASW} that a union of \emph{three} grids can determine a uniformly discrete dense forest with visibility $O\left(\epsilon^{-5-\eta}\right)$ for any $\eta>0$ \emph{in the plane}. Such union $\mathfrak{B}$ can  for instance be defined as 
\begin{equation}\label{defunifdiscrgridforest}
\mathfrak{B}\;=\; \bm{\Lambda}_1\cup\left(\bm{x}+\bm{\Lambda}_2\right)\cup\left(\bm{\bm{y}+\Lambda}_3\right), 
\end{equation} 
where 
\begin{equation*}
\bm{\Lambda}_1=\Z^2, \qquad \bm{\Lambda}_2=\begin{pmatrix}\gamma & \alpha\\0& 1\end{pmatrix}\cdot \Z^2\quad\textrm{and}\qquad  \bm{\Lambda}_3=\begin{pmatrix}1 & 0\\ \beta& \delta\end{pmatrix}\cdot \Z^2
\end{equation*} 
with 
\begin{equation}\label{defcoefmatr}
\alpha=\sqrt{2}, \qquad \beta=3-\sqrt{2}+\sqrt{3}-\sqrt{6}, \qquad \gamma=\sqrt{3} \quad\textrm{and}\qquad \delta=-3+\sqrt{6}.
\end{equation}
Furthermore, $\bm{x}, \bm{y}\in\R^2$ can be chosen such that 
\begin{equation}\label{deftransl}
\bm{x}\not\in \overline{\bm{\Lambda}_2-\bm{\Lambda}_1}, \qquad \bm{y}\not\in \overline{\bm{\Lambda}_3-\bm{\Lambda}_1} \qquad \textrm{and}\qquad \bm{y}-\bm{x}\not\in \overline{\bm{\Lambda}_3-\bm{\Lambda}_2}
\end{equation}
(which holds for generic choices of $\bm{x}$, $\bm{y}\in\R^2$).

Two comments are in place. 

On the one hand, the values in~\eqref{defcoefmatr} are particular examples of parameters $\alpha, \beta, \gamma$ and $\delta$  satisfying certain general conditions of an algebraic and Diophantine nature explicitly stated in~\cite[\S 6.1]{ASW}. The following questions should be rekated to the comment made before the statement of Problem~\ref{udtmanifolds} ~:

\begin{pb}
Can the construction of the uniformly discrete dense forest defined in~\eqref{defunifdiscrgridforest} be generalised to higher dimensions?
\end{pb}

\begin{pb}\label{defvisiunifdicrforgrid}
Can the visibility bound $O\left(\epsilon^{-5-\eta}\right)$ for any $\eta>0$ be improved for the dense forest~\eqref{defunifdiscrgridforest}?
\end{pb}

We expect the answer to the question raised in Problem~\ref{defvisiunifdicrforgrid} to be almost certainly positive.

On the other hand, this construction is not explicit only because of the generic choice of the vectors $\bm{x}, \bm{y}\in\R^2$ in~\eqref{deftransl}.

\begin{pb}
To find explicit examples of vectors $\bm{x}, \bm{y}\in\R^2$ such that the relations in~\eqref{deftransl} hold.
\end{pb}

\subsection{Cut--and--Project Sets} We briefly recall first the definition of a cut--and--project set, referring the reader to~\cite{BG, BM00, Me94, Moody, Se95} for further details. These are  aperiodic but structured subsets of Euclidean space providing models for the mathematical theory of quasicystals.

Let $n\ge 2$. Decompose $\R^n$ as the direct sum of two subspaces, $S_{phys}$ (the \emph{physical space}) and $S_{int}$ (the \emph{internal space}). In particular, $n=\dim \left(S_{phys}\right)+\dim \left(S_{int}\right)$. Let $\pi_{phys}~: \R^n\rightarrow S_{phys}$ and $\pi_{int}~: \R^n\rightarrow S_{int}$ denote the projections corresponding to the direct sum decomposition. Fix a grid $\Lambda\subset\R^n$ obtained as the translate of a full rank lattice and let $\mathcal{W}\subset S_{int}$ be a bounded set (referred to as the \emph{window}). Then, $$\mathfrak{L}\left(\Lambda, \mathcal{W}\right)\;=\; \pi_{phys}\left(\Lambda\cap \pi_{int}^{-1}\left(\mathcal{W}\right)\right)$$ is the cut--and--project set   with dimensions $\left(\dim S_{phys}, n\right)$  associated  to the grid $\Lambda$, to the window $\mathcal{W}$ and to the direct sum decomposition $\R^n=S_{phys}\oplus  S_{int}$. See Figure~\ref{candp} for an illustration of a  cut--and--project set with dimensions $(1,2)$.

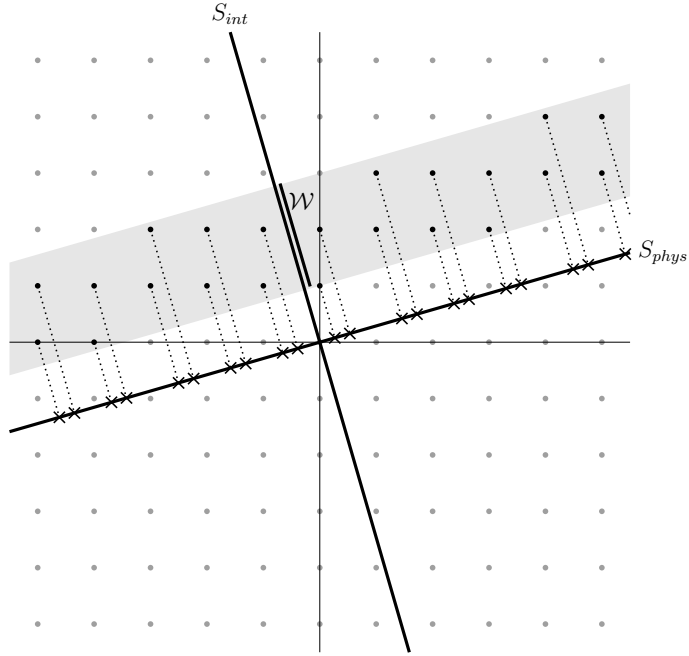
\begin{figure}[h!]
\begin{center}
\scalebox{0.75}{
  \begin{tikzpicture}    

	\fill[gray!20] (-5.5,-5.5*0.28867513459+1)--(5.5,5.5*0.28867513459+1)  -- (5.5, 5.5*0.28867513459+3) -- (-5.5, -5.5*0.28867513459+3)--cycle;  
  
	\fill[gray!75] (0,0) circle(0.05);
	\fill (0,1) circle(0.05); 
	\fill (0,2) circle(0.05); 
	\fill[gray!75] (0,3) circle(0.05); 
	\fill[gray!75] (0,4) circle(0.05); 
	\fill[gray!75] (0,5) circle(0.05); 
	\fill[gray!75] (0,-1) circle(0.05); 
	\fill[gray!75] (0,-2) circle(0.05); 
	\fill[gray!75] (0,-3) circle(0.05); 
	\fill[gray!75] (0,-4) circle(0.05); 
	\fill[gray!75] (0,-5) circle(0.05); 
	\fill[gray!75] (1,0) circle(0.05); 
	\fill[gray!75] (1,1) circle(0.05); 
	\fill[gray!75] (1,-1) circle(0.05); 
	\fill (1,2) circle(0.05); 
	\fill[gray!75] (1,-2) circle(0.05); 
	\fill (1,3) circle(0.05); 
	\fill[gray!75] (1,-3) circle(0.05); 
	\fill[gray!75] (1,4) circle(0.05); 
	\fill[gray!75] (1,-4) circle(0.05); 
	\fill[gray!75] (1,5) circle(0.05);  
	\fill[gray!75] (1,-5) circle(0.05); 
	\fill[gray!75] (2,0) circle(0.05); 
	\fill[gray!75] (2,1) circle(0.05); 
	\fill[gray!75] (2,-1) circle(0.05); 
	\fill (2,2) circle(0.05); 
	\fill[gray!75] (2,-2) circle(0.05); 
	\fill (2,3) circle(0.05); 
	\fill[gray!75] (2,-3) circle(0.05); 
	\fill[gray!75] (2,4) circle(0.05); 
	\fill[gray!75] (2,-4) circle(0.05); 
	\fill[gray!75] (2,5) circle(0.05);  
	\fill[gray!75] (2,-5) circle(0.05);         
	\fill[gray!75] (3,0) circle(0.05); 
	\fill[gray!75] (3,1) circle(0.05); 	
	\fill[gray!75] (3,-1) circle(0.05); 
	\fill (3,2) circle(0.05); 
	\fill[gray!75] (3,-2) circle(0.05); 
	\fill (3,3) circle(0.05); 
	\fill[gray!75] (3,-3) circle(0.05); 
	\fill[gray!75] (3,4) circle(0.05); 
	\fill[gray!75] (3,-4) circle(0.05); 
	\fill[gray!75] (3,5) circle(0.05);  
	\fill[gray!75] (3,-5) circle(0.05); 	 
	\fill[gray!75] (4,0) circle(0.05); 
	\fill[gray!75] (4,1) circle(0.05); 	
	\fill[gray!75] (4,-1) circle(0.05); 
	\fill[gray!75] (4,2) circle(0.05); 
	\fill[gray!75] (4,-2) circle(0.05); 
	\fill (4,3) circle(0.05); 
	\fill[gray!75] (4,-3) circle(0.05); 
	\fill (4,4) circle(0.05); 
	\fill[gray!75] (4,-4) circle(0.05); 
	\fill[gray!75] (4,5) circle(0.05);  
	\fill[gray!75] (4,-5) circle(0.05); 	
	\fill[gray!75] (5,0) circle(0.05); 
	\fill[gray!75] (5,1) circle(0.05); 	
	\fill[gray!75] (5,-1) circle(0.05); 
	\fill[gray!75] (5,2) circle(0.05); 
	\fill[gray!75] (5,-2) circle(0.05); 
	\fill (5,3) circle(0.05); 
	\fill[gray!75] (5,-3) circle(0.05); 
	\fill (5,4) circle(0.05); 
	\fill[gray!75] (5,-4) circle(0.05); 
	\fill[gray!75] (5,5) circle(0.05);  
	\fill[gray!75] (5,-5) circle(0.05); 	
	\fill (-5,0) circle(0.05); 
	\fill (-5,1) circle(0.05); 	
	\fill[gray!75] (-5,-1) circle(0.05); 
	\fill[gray!75] (-5,2) circle(0.05); 
	\fill[gray!75] (-5,-2) circle(0.05); 
	\fill[gray!75] (-5,3) circle(0.05); 
	\fill[gray!75] (-5,-3) circle(0.05); 
	\fill[gray!75] (-5,4) circle(0.05); 
	\fill[gray!75] (-5,-4) circle(0.05); 
	\fill[gray!75] (-5,5) circle(0.05);  
	\fill[gray!75] (-5,-5) circle(0.05); 	
	\fill (-4,0) circle(0.05); 
	\fill (-4,1) circle(0.05); 	
	\fill[gray!75] (-4,-1) circle(0.05); 
	\fill[gray!75] (-4,2) circle(0.05); 
	\fill[gray!75] (-4,-2) circle(0.05); 
	\fill[gray!75] (-4,3) circle(0.05); 
	\fill[gray!75] (-4,-3) circle(0.05); 
	\fill[gray!75] (-4,4) circle(0.05); 
	\fill[gray!75] (-4,-4) circle(0.05); 
	\fill[gray!75] (-4,5) circle(0.05);  
	\fill[gray!75] (-4,-5) circle(0.05); 	
	\fill[gray!75] (-3,0) circle(0.05); 
	\fill (-3,1) circle(0.05); 	
	\fill[gray!75] (-3,-1) circle(0.05); 
	\fill (-3,2) circle(0.05); 
	\fill[gray!75] (-3,-2) circle(0.05); 
	\fill[gray!75] (-3,3) circle(0.05); 
	\fill[gray!75] (-3,-3) circle(0.05); 
	\fill[gray!75] (-3,4) circle(0.05); 
	\fill[gray!75] (-3,-4) circle(0.05); 
	\fill[gray!75] (-3,5) circle(0.05);  
	\fill[gray!75] (-3,-5) circle(0.05); 		
	\fill[gray!75] (-2,0) circle(0.05); 
	\fill (-2,1) circle(0.05); 	
	\fill[gray!75] (-2,-1) circle(0.05); 
	\fill (-2,2) circle(0.05); 
	\fill[gray!75] (-2,-2) circle(0.05); 
	\fill[gray!75] (-2,3) circle(0.05); 
	\fill[gray!75] (-2,-3) circle(0.05); 
	\fill[gray!75] (-2,4) circle(0.05); 
	\fill[gray!75] (-2,-4) circle(0.05); 
	\fill[gray!75] (-2,5) circle(0.05);  
	\fill[gray!75] (-2,-5) circle(0.05); 	    
	\fill[gray!75] (-1,0) circle(0.05); 
	\fill (-1,1) circle(0.05); 	
	\fill[gray!75] (-1,-1) circle(0.05); 
	\fill (-1,2) circle(0.05); 
	\fill[gray!75] (-1,-2) circle(0.05); 
	\fill[gray!75] (-1,3) circle(0.05); 
	\fill[gray!75] (-1,-3) circle(0.05); 
	\fill[gray!75] (-1,4) circle(0.05); 
	\fill[gray!75] (-1,-4) circle(0.05); 
	\fill[gray!75] (-1,5) circle(0.05);  
	\fill[gray!75] (-1,-5) circle(0.05); 	
	
	\draw (0, -5.5)--(0, 5.5) ;
	\draw (-5.5,0)--(5.5,0) ;
	\draw[ultra thick] (-5.5, -5.5*0.28867513459)--(5.5, 5.5*0.28867513459);
	\draw[ultra thick] (-5.5*0.28867513459, 5.5)--(5.5*0.28867513459, -5.5);
	\draw [ultra thick] (-0.27+0.1, 0.960768923+0.1*0.28867513459)--(2.9*-0.277350098+0.1, 2.9*0.960768923+0.1*0.28867513459);

	\node [above] at (-5.5*0.28867513459, 5.5) {$S_{int}$};  	
	\node [right] at ( 5.5, 5.5*0.28867513459) {$S_{phys}$};  	
	\node [left] at (0.05,2.5) {$\mathcal{W}$};

	\draw[dotted, thick] (0,1)--(0.2665, 0.0769);
	\draw[dotted, thick] (0,2)--(0.5329, 0.1538);
	\draw[dotted, thick] (1,2)--(1.4560, 0.4203);
	\draw[dotted, thick] (1,3)--(1.7225, 0.4972);
	\draw[dotted, thick] (2,3)--(2.6456, 0.7637);
	\draw[dotted, thick] (2,2)--(2.3791, 0.6868);		
	\draw[dotted, thick] (3,2)--(3.3022, 0.9533);	
	\draw[dotted, thick] (3,3)--(3.5686, 1.0302);
	\draw[dotted, thick] (4,4)--(4.7582, 1.3736);
	\draw[dotted, thick] (4,3)--(4.4917, 1.2966);
	\draw[dotted, thick] (5,3)--(5.4148, 1.5631);	
	\draw[dotted, thick] (5,4)--(5.5, 2.25);								
	\draw[dotted, thick] (-1,1)--(-0.6566, -0.1895);
	\draw[dotted, thick] (-1,2)--(-0.3914, -0.1126);
	\draw[dotted, thick] (-2,1)--(-1.5797, -0.456);
	\draw[dotted, thick] (-2,2)--(-1.3132, -0.3791);		
	\draw[dotted, thick] (-3,1)--(-2.5028, -0.7225);
	\draw[dotted, thick] (-3,2)--(-2.2363, -0.6456);		
	\draw[dotted, thick] (-4,0)--(-3.6923, -1.0659);	
	\draw[dotted, thick] (-4,1)--(-3.4258, -0.9886);		
	\draw[dotted, thick] (-5,0)--(-4.6154, -1.3323);	
	\draw[dotted, thick] (-5,1)--(-4.3489, -1.2554);					

\draw plot[mark=x, mark options={color=black, scale=2,  thick}] coordinates {(0.2665, 0.0769) (0.5329, 0.1538) (1.4560, 0.4203) (1.7225, 0.4972) (2.6456, 0.7637) (2.3791, 0.6868)(3.3022, 0.9533)(3.5686, 1.0302)(4.7582, 1.3736)(4.4917, 1.2966)(5.4148, 1.5631) (-0.6566, -0.1895) (-0.3914, -0.1126) (-1.5797, -0.456) (-1.3132, -0.3791) (-2.5028, -0.7225) (-2.2363, -0.6456) (-3.6923, -1.0659) (-3.4258, -0.9886) (-4.6154, -1.3323) (-4.3489, -1.2554)};
	
  \end{tikzpicture}
  }
  \end{center}
  \caption{Representation of a cut--and--project set with dimensions (1,2) obtained from the lattice $\Z^2$, the physical space $S_{phys}: y=x/(2\sqrt{3})$ and the internal space $S_{int}=S_{phys}^{\perp}$. The window $\mathcal{W}$ is an interval of length 2 in the internal space and the cut--and--project set is the set of crosses in the physical space.} 
  \label{candp}
\end{figure}

Cut--and--project sets are uniformly discrete. Theorem 1.1 in~\cite{ASW} establishes that they can never be dense forests as they always admit a vacant strip (in the sense defined in~\S\ref{secgrid}). Furthermore, Theorem 1.2 in~\cite{SW} shows that a finite union of cut--and--project sets can never be a Danzer set. In dimension $n=2$ however, a finite union of such sets can form a uniformly discrete dense forest. This is proved in~\cite[Theorem 2.5]{ASW}, and we   describe the underlying construction in order to state natural questions emerging from it.\\

Set $\T^3=\R^3\backslash\Z^3$ and denote by $\pi~:\R^3\rightarrow\T^3$ the natural projection. A \emph{piecewise linear unavoidable section} $\mathcal{S}\subset\T^3$ is a disjoint union of finitely many projections, under $\pi$, of line segments in $\R^3$ which satisfies the following intersection property~: for every $\bm{x}\in\R^3$ and every two--dimensional rational subspace $R\subset\R^3$, it holds that $\mathcal{S}\cap\pi\!\left(R+\bm{x}\right) \neq\emptyset$. It is not hard to see that such sections exist (see~\cite[Theorem 2.4]{ASW} for explicit examples).

Fix  a  piecewise linear unavoidable section $\mathcal{S}\subset\T^3$ and $\bm{x}_0\in\T^3$. Consider a $2$--dimensional subspace $V\subset\R^3$ which is totally irrational in the sense that it does not contain any rational line. Assume furthermore that $V$ is transverse to $\mathcal{S}$ in the sense that it does not contain any of the line segments the projections of which define the set $\mathcal{S}$. Set 
\begin{equation}\label{candpudforest}
\mathfrak{C}\;=\;\left\{\bm{v}\in V\; :\; \bm{x}_0+\pi(\bm{v})\in\mathcal{S}\right\}.
\end{equation} 
Then, $\mathfrak{C}$ is a uniformly discrete dense forest in the subspace $V$ identified with $\R^2$. As a matter of fact, $\mathfrak{C}$ is a finite union of cut--and--project sets, all with dimensions $(2,3)$ and all sharing the same physical space $V$. See the proof of Theorem 2.5 in~\cite{ASW} for a justification of these claims. 

The following  problem also appears in~\cite[\S 8]{ASW}~:

\begin{pb}
Determine a visiblity bound for the planar uniformly discrete dense forest defined by the union of cut--and--project sets~\eqref{candpudforest}.
\end{pb}

As the previous construction only deals with the planar case, it is natural to seek for a generalisation to higher dimensions.

\begin{pb}
Can a finite union of cut--and--project sets be a (possibly uniformly discrete) dense forest in dimension $n\ge 3$?
\end{pb}

\subsection{Point Sets Emerging from a Tiling obtained by a Primitive Substitution System} A countable collection of $n$--dimensional polytopes $\left\{P_i\right\}_{i\ge 0}$ is said to tile the space $\R^n$ if their union covers $\R^n$ and if any two distinct elements in this collection only intersect at their boundary (if they intersect at all). The polytopes are then said to be the tiles and the tiling is said to be polygonal. Given a finite set of $n$--dimensional polytopes, say $P_1, \dots, P_s$, there is a natural way to generate a tiling of the entire space by ``dissection and inflation''. We briefly describe this process, referring the reader to~\cite{BG} for further details.

Assume the existence of a \emph{substitution}, namely a map $S$ which to a given polytope $P_i$ associates a tiling of $P_i$ by isometric copies of the scaled tiles $\alpha P_1, \dots, \alpha P_s$. Here, $\alpha>0$ is a parameter whose inverse is referred to as the \emph{inflation constant}. Clearly, the definition of $S$ can be extended so as to tile any finite union of isometric copies of the polytopes $P_i$'s by isometric copies of  the scaled polytopes $\alpha P_1, \dots, \alpha P_s$. By applying the mapping $S$ successively, and by inflating back each time the smaller tiles  to their initial size, one can thus  cover bigger and bigger regions of the Euclidean space $\R^n$. It can furthermore be shown that this process converges, in the limit, to a tiling of the whole space. Such a tiling is then called a \emph{polygonal substitution tiling}. It is \emph{primitive} if the \emph{substitution matrix} $M_S$ is primitive in the usual sense (i.e.~there exists $k\ge 1$ such that $M_S^k$ has strictly positive entries). Here, $M_S^k$ is the $s\times s$ matrix whose $(i,j)$ entry is the number of times an isometric copy of $\alpha P_i$ is used to tile $S(P_j)$. Primitivity is a natural concept in this context as it renders the idea that the substition is ``irreducible'' in the sense that it does not emerge from a smaller one.  See Figure~\ref{substilpol} for an illustration. \\

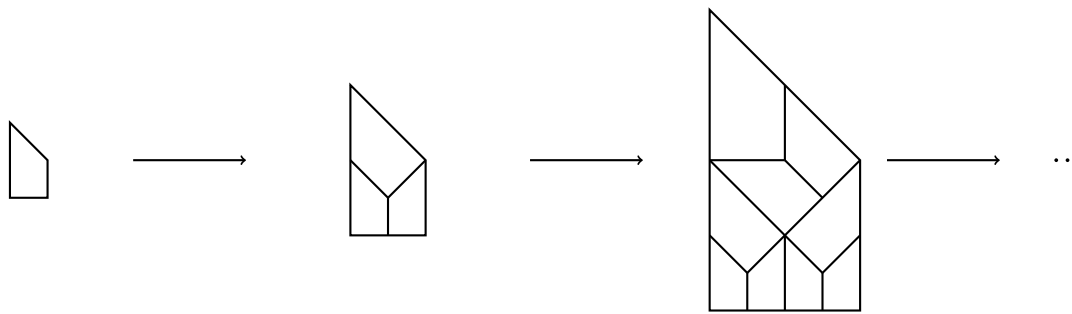
\begin{figure}[h!]
\begin{center}
\scalebox{0.5}{

\hspace{-20mm}
\begin{subfigure}{.45\textwidth}
  \centering
  
  \begin{tikzpicture}    
  
  	\draw[ultra thick] (-0.5,-0.5)-- (0.5,-0.5)-- (0.5,0.5) -- (-0.5,1.5)--cycle;
  		
  \end{tikzpicture}
\end{subfigure}
\hspace{-30mm}
\begin{subfigure}{.45\textwidth}
  \centering

  \begin{tikzpicture}    
  
  	\draw[->, ultra thick] (0,0)--(3,0);	
  	  	  	
  \end{tikzpicture}
\end{subfigure}
\hspace{-20mm}
\begin{subfigure}{.45\textwidth}
  \centering

  \begin{tikzpicture}    
  
  	\draw[ultra thick] (2*-0.5,2*-0.5)-- (2*0.5,2*-0.5)-- (2*0.5,2*0.5) -- (2*-0.5,2*1.5)--cycle;  		
  	\draw[ultra thick] (0,0)--(2*-0.5, 2*0.5);
  	\draw[ultra thick] (0,0)--(2*0.5, 2*0.5);
  	\draw[ultra thick] (0,0)--(0, 2*-0.5);  	  	
  	  	  	
  \end{tikzpicture}
\end{subfigure}
\hspace{-20mm}
\begin{subfigure}{.45\textwidth}
  \centering

  \begin{tikzpicture}    
  
  	\draw[->, ultra thick] (0,0)--(3,0);		
  	  	  	
  \end{tikzpicture}
\end{subfigure}
\hspace{-20mm}
\begin{subfigure}{.45\textwidth}
  \centering
  \begin{tikzpicture}    
    	\draw[ultra thick] (4*-0.5,4*-0.5)-- (4*0.5,4*-0.5)-- (4*0.5,4*0.5) -- (4*-0.5,4*1.5)--cycle;  		
  	\draw[ultra thick] (0,0)--(4*-0.5, 4*0.5);
  	\draw[ultra thick] (0,0)--(4*0.5, 4*0.5);
  	\draw[ultra thick] (0,0)--(0, 4*-0.5);  	  
  	\draw[ultra thick] (0,4*0.5)--(4*-0.5, 4*0.5);  	   		  	
  	\draw[ultra thick] (0,4*0.5)--(0, 8*0.5);  	   	
  	\draw[ultra thick] (0, 4*0.5)--(2*0.5, 1);  	
  	\draw[ultra thick] (1, -4*0.5)--(1, -2*0.5);  	   
  	\draw[ultra thick] (-1, -4*0.5)--(-1, -2*0.5);  	   
  	\draw[ultra thick] (-1, -2*0.5)--(0,0);    	  	  	  
  	\draw[ultra thick] (1, -2*0.5)--(0,0); 	  	
  	\draw[ultra thick] (1, -2*0.5)--(2,0); 	  
  	\draw[ultra thick] (-1, -2*0.5)--(-2,0); 	    	
  \end{tikzpicture}
\end{subfigure}
\hspace{-20mm}
\begin{subfigure}{.45\textwidth}
  \centering

  \begin{tikzpicture}    
  
  	\draw[->, ultra thick] (0,0)--(3,0);		
	\fill (4.5,0) circle(0.05);
	\fill (4.8,0) circle(0.05);
	\fill (5.1,0) circle(0.05);		
  \end{tikzpicture}
\end{subfigure}

}
\end{center}
  \caption{Illustration of a substitution tiling obtained from two different polygonal tiles.} 
  \label{substilpol}
\end{figure}

Uniformly discrete sets can be naturally obtained from a substitution tiling in at least two different ways~: either by choosing a point always in the same position in each tile, or else by considering the set of vertices of this tiling. Denote by $\Sigma\subset\R^n$ a point set obtained by either of these processes. Theorem 1.1 in~\cite{SW} shows that $\Sigma$ is never a Danzer set, and the same proof establishes that it is not a dense forest either. As pointed out to the author by Yaar Solomon, these results can also be  generalised to the case of so--called \emph{multiscale substitution tilings} recently introduced by Smilanski and Solomon~\cite{SmS}.

The following two problems appear in~\cite[\S 7.1]{SW}. 

\begin{pb}
Consider a finite number of primitive substitution tilings in $\R^n$ and the  union of the resulting sets $\Sigma$ obtained as above. Can this union be a Danzer set or a dense forest?
\end{pb}

The  definition of a substitution tiling can be generalised to the case that the tiles are not necessarily polygonal but fractal (see~\cite{BSol} for further details). In this case however, the proof of Theorem 1.1 in~\cite{SW} does not apply anymore.

\begin{pb}
Consider a primitive substitution tiling obtained from a fractal tiling. Define a point set by choosing a point always in the same position in each tile. Can this point set be a Danzer set or a dense forest?
\end{pb}

\subsection{Holonomy Vectors of Saddle Connections on Translation Surfaces} Informally, a \emph{translation surface} is the topological space obtained after identifying opposite parallel sides of a planar polygon. More formally, it is a Riemann surface equipped with a holomorphic 1--form that can be realized this way~: let $P_1, \dots, P_s$ be a collection of    polygons in $\R^2$. Assume that for every edge $\bm{E}_i$ of one of these polygons, there exists a different edge $\bm{E}_i$ of the polygons such that $\bm{E}_j = \bm{E}_i + \bm{v}_{ji}$ for some non--zero vector $\bm{v}_{ji}$ (in particular, that $\bm{v}_{ij}=-\bm{v}_{ji}$). The translation surface resulting from this configuration is the closed quotient space obtained by identifying every edge $\bm{E}_i$ with its associated edge $\bm{E}_j$ through the translation $\bm{x}\in\R^2\mapsto \bm{x}+\bm{v}_{ij}\in\R^2$. See Figure~\ref{transsfce} for an illustration and~\cite{HS, M, W} and the references within for more details on this concept.

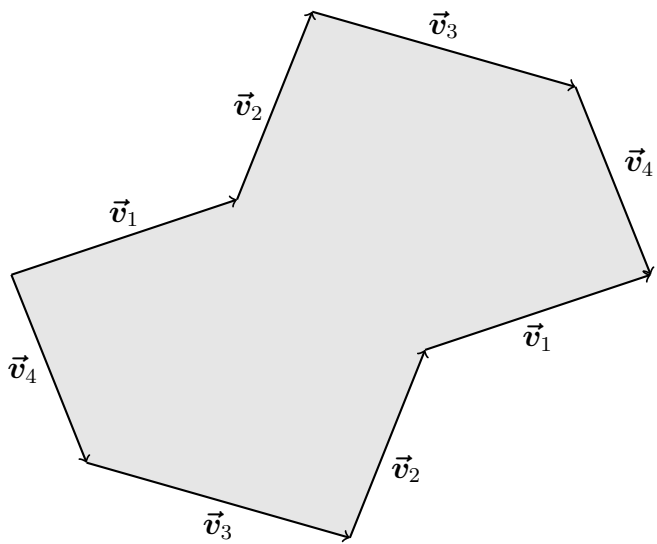
\begin{figure}[h!]
\begin{center}
\scalebox{1}{
  \begin{tikzpicture}    
  
	\fill[gray!20] (-1,1)--(0,3.5)  -- (3.5, 2.5) -- (4.5, 0) -- (1.5,-1)-- (0.5,-3.5) -- (-3,-2.5)-- (-4,0)--cycle;
        
    \draw[->, thick] (-1,1)--(0,3.5) ;
    \draw[->, thick] (0.5,-3.5)--(1.5,-1) ;
    \draw[->, thick] (-4,0)--(-1,1) ;
    \draw[->, thick] (1.5,-1)--(4.5, 0) ;    
    \draw[->, thick] (0,3.5)--(3.5, 2.5) ;    
    \draw[->, thick] (-4,0)--(-3, -2.5) ;               
    \draw[->, thick] (-3,-2.5)--(0.5, -3.5) ;         
    \draw[->, thick] (3.5,2.5)--(4.5, 0) ;         
    
    \node [above] at (-2.5, 0.5) {$\bm{\vec{v}}_1$};
    \node [below] at (3, -0.5) {$\bm{\vec{v}}_1$};
    \node [above, left] at ( -0.5, 2.25) {$\bm{\vec{v}}_2$};    
    \node [below] at ( 1.25, -2.25) {$\bm{\vec{v}}_2$};    
    \node [above] at ( 1.75, 3) {$\bm{\vec{v}}_3$};   
    \node [right] at ( 4, 1.5) {$\bm{\vec{v}}_4$};     
    \node [below] at ( -1.25, -3) {$\bm{\vec{v}}_3$};         
    \node [left] at ( -3.5, -1.25) {$\bm{\vec{v}}_4$};

  \end{tikzpicture}
  }
  \end{center}
  \caption{Representation of a translation surface obtained by identifying by translation the opposite parallel edges of a polygon.} 
  \label{transsfce}
\end{figure}

The metric induced on this quotient space is  flat outside the singularities, where a \emph{singularity} is defined as the image of a vertex around which  the sum of the angles of the polygons mapping to it is not $2\pi$. (In Figure~\ref{transsfce} for instance, there is a single singularity with angle $6\pi$.) A \emph{saddle connection} is then defined as a geodesic between two (not necessarily distinct) singularities which does not cross any other singularity. The corresponding \emph{holonomy vector} is obtained by ``devlopping'' in the natural way the saddle connection in the plane thanks to the local coordinates~: it is therefore a vector with the same length and direction as the saddle connection. For an equivalent definition relying on  the Riemann surface structure associated to a holomorphic 1--form, see, e.g., \cite[\S 1]{CR}. \\

A fundamental result due to Masur~\cite{HM} states that the set of holonomy vectors of a translation surface  has always finite density. It nevertheless need not be \emph{relatively dense}, which is a necessary condition for a point set to be a dense forest or a Danzer set\footnote{A set is relatively dense if there exists $R>0$ such that any ball with radius $R$ intersects it non--trivially.}. See~\cite{Wu} for counterexamples justifying this claim. The following problem is attributed to Barak Weiss. 

\begin{pb}[Barak Weiss]
Determine necessary and sufficient conditions for the set of holonomy vectors of a translation surface to be relatively dense and/or uniformly discrete.
\end{pb}

To be precise, Weiss' question is about determining  conditions under which the set under consideration is Delone; that is, \emph{both} relatively dense and uniformly discrete. Several partial results are provided in~\cite{Wu}. Since a set of holonomy vectors enjoys some nice distribution properties at least in a probabilistic sense (see~\cite{CR} for a precise statement), it is natural to related it to the Danzer problem and to the construction of dense forests.

The following statement and its (sketched) proof were communicated to the author by Barak Weiss~:

\begin{thm}
A  set of holonomy vectors can \emph{never} be a Danzer set.
\end{thm}

\begin{proof}[Sketch of the proof]
A set of holonomy vectors of a surface has an equivariance property with respect to the action of the group $SL(2,\R)$ on the space of surfaces and on the plane by linear transformation. Namely, if $\mathcal{M}$ is a translation surface, $g$ is an element of $SL(2,\R)$ and $hol(\mathcal{M})$ is the set of holonomy vectors of $\mathcal{M}$, then $g\cdot hol(\mathcal{M}) = hol(g\mathcal{M})$. A surface is called \emph{horizontally completely periodic} if it is the union of horizontal cylinders; that is there are finitely many axis parallel rectangles which, when glued to each other give the surface, and such that each of the vertical sides of a rectangle are identified with each other, so that each rectangle closes up to become a cylinder which is horizontally periodic. It is easy to see that the holonomy set of a horizontally completely periodic surface has an infinite strip missing the holonomy set. Namely if $h>0$ is the minimal height of a horizontal cylinders, then there will be no holonomy vectors in the infinite strip $\{(x,y) : 0<y<h\}$. In particular, a horizontally completely periodic surface has holonomy set which is neither a Danzer set, nor a dense forest.

In~\cite{SmWe}, it is established that the orbit-closure under $SL(2,\R)$ of any surface contains a horizontally completely periodic surface (in fact, this is established in a stronger form, namely when considering the orbit closure under the upper triangular unipotent matrices). 

The space of surfaces can be equipped with a topology such that when a sequence $\left(\mathcal{M}_n\right)_{n\ge 1}$ converges to $\mathcal{M}$ in the space of surfaces, this implies that the sequence $\left(hol(\mathcal{M}_n)\right)_{n\ge 1}$ converges to $hol(\mathcal{M})$ in the Chabauty--Fell topology (defined in~\eqref{fell}). 

From all of this the theorem can be inferred as follows~: Let $\mathcal{M}_0$ be any translation surface and let $g_n \in SL(2,R)$ be such that the sequence $\left(\mathcal{M}_n\right)_{n\ge 1} = \left(g_n\cdot\mathcal{M}_0\right)_{n\ge 1}$ converges to  a horizontally completely periodic surface $\mathcal{M}$ in the topology on the collection of translation surfaces. Then, the sequence $\left(hol(\mathcal{M}_n)\right)_{n\ge 1}$ converges to $hol(\mathcal{M})$ in the Chabauty--Fell topology. Since $hol(\mathcal{M})$ misses an infinite strip and $hol(\mathcal{M}_n)$ tends to $hol(\mathcal{M})$, for any $T\ge 1$, one can find $n\ge 1$ so that $hol(\mathcal{M}_n)$ misses a convex set of area $> T$. Pulling back by $g_n^{-1}$ shows that the same holds for $\mathcal{M}_0$. 
\end{proof}

This leaves open the following question~:
\begin{pb}
Assume that the set of holonomy vectors of a translation surface is relatively dense. Can it be a dense forest?
\end{pb}

\end{document}